%% file: main.tex
\documentclass[10pt, conference, letterpaper]{IEEEtran}

\usepackage[utf8]{inputenc}
\usepackage{bm}
\usepackage{algorithmicx}
\usepackage[noend]{algpseudocode}
\usepackage{algorithm}
\usepackage{latexsym}
\usepackage{amssymb}
\usepackage{amsbsy}
\usepackage{amsmath}
\usepackage{graphicx}
\usepackage[caption=false]{subfig}
\usepackage{enumerate}
\usepackage{url}
\usepackage{romannum}
\usepackage{capt-of}
\usepackage{multirow}

\usepackage{tikz}

\usepackage{hhline}

\newcommand{\idc}{\bm{1}}
\newcommand{\N}{\mathcal{N}}

\newcommand{\M}{\mathcal{M}}

\newcommand{\Ex}{\mathbb{E}}
\newcommand{\pr}{\mathbb{P}}
\newcommand{\bP}{\mathbf{P}}

\newcommand{\p}{\mathbf{P}}
\newcommand{\PI}{\boldsymbol{\pi}}

\newcommand{\bX}{\mathcal{X}}

\newcommand{\x}{\mathbf{x}}

\newcommand{\R}{\mathbb{R}}

\newtheorem{theorem}{Theorem}
\newtheorem{definition}{Definition}
\newtheorem{lemma}{Lemma}

\newtheorem{proposition}{Proposition}

\newtheorem{conjecture}{Conjecture}

\allowdisplaybreaks

\title{Simulated Annealing for Optimal Resource Allocation in Wireless Networks with Imperfect Communications}

\author{%
  \IEEEauthorblockN{Jaewook Kwak}
  \and
  \IEEEauthorblockN{Ness B. Shroff}
}

\begin{document}

\maketitle

\begin{abstract}

Simulated annealing (SA) method has had significant recent success in designing distributed control algorithms for wireless networks.
These SA based techniques formed the basis of new CSMA algorithms and gave rise to the development of numerous variants to achieve the best system performance accommodating different communication technologies and more realistic system conditions.
However, these algorithms do not readily extend to networks with noisy environments, as unreliable communication prevents them from gathering the necessary system state information needed to execute the algorithm.
In recognition of this challenge, we propose a new SA algorithm that is designed to work more robustly in networks with communications that experience frequent message drops.
The main idea of the proposed algorithm is a novel coupling technique that takes into account the external randomness of message passing failure events as a part of probabilistic uncertainty inherent in stochastic acceptance criterion of SA. 
As a result, the algorithm can be executed even with partial observation of system states, which was not possible under the traditional SA approach.
We show that the newly proposed algorithm finds the optimal solution almost surely under the standard annealing framework while offering significant performance benefits in terms of its computational speed in the presence of frequent message drops.


\end{abstract}

\let\ORIGthefootnote\thefootnote
\renewcommand{\thefootnote}{}
\footnotetext{The authors are with the Department of Electrical and Computer Engineering, The Ohio State University, Columbus, OH, Email: \{kwak.84, shroff.11\}@osu.edu. This work has in part been funded by: a grant from the Defense Thrust Reduction Agency (DTRA) HDTRA1-14-1-0058; and from the National Science Foundation grants CNS-1446582 and CNS-1518829.}

\input{1_intro}

\input{2_related}

\input{3_prelim}

\input{5_challenge}

\input{6_algorithm}

\input{7_simulation}

\input{8_conclusion}

\input

\small
\bibliographystyle{abbrv}
\bibliography{ref_all}
\normalsize

\input{appendix}

\end{document}

%% file: 1_intro.tex
\section{Introduction}

In modern wireless network systems, many network functionalities involve solving complex network-wide decision problems.
Example network problems include media access control, routing optimization, resource allocation, and QoS provisioning in wireless networks, etc.
A common goal pursued in these problems is to achieve the desired performance objective by seeking the best configuration of a set of system parameters. 
This requirement naturally leads to form a certain combinatorial optimization problem to be solved in distributed settings. However, these problems are often very difficult and high-dimensional such that their complexity grows rapidly with the size of the network.

In this paper, we consider an important class of optimization problems that are primarily motivated by resource allocation and link scheduling problems in wireless networks. 
A classical example of such problems is the max-weight or weighted sum rate maximization problems, which serves as a basis for many resource management and network design problems.
These problems are typically difficult to solve, and are in general known to be NP-hard even in the simple binary capacity model.
In addition, emerging wireless communication technologies employ increasingly complex adaptive modulation and coding techniques, which further exacerbate the complexity of these problems. 
We focus on a class of NP-hard type resource allocation problems which are often intractable to solvein an efficient way and even in a centralized manner.

The solution methodology we develop in this paper is based on the classical Simulated Annealing (SA) method \cite{kirkpatrick1983optimization}, which is a randomized technique for approximating the optimum for a given objective function. 
The algorithmic procedure of SA is intuitive and simple. 
In each step, a trial state is randomly generated and its performance objective is evaluated. If the trial state improves the objective, the current state is replaced by the new state. If the objective of the trial state is not better than that of current one, the trial is accepted or rejected based on a certain probabilistic criterion.
The advantages of SA are the relative ease of implementation and the ability to provide good solutions with provable guarantees for any arbitrary systems and objective functions.
Since SA is such a ubiquitous method, it has found wide-spread applications in various engineering problems \cite{ogbu1990application, tian1999application, gong2012community}.

An integral step needed to realize SA in practical systems is the correct evaluation of the performance objective on each system state, or at least the performance differential between the current and trial states.
In a distributed network where the performance objective is dependent on multiple system variables across different nodes, the task of measuring the objective differentials can be done by implementing a proper message passing mechanism.
For wireless resource allocation problems, to which SA is applied, most works implicitly assume that these message exchanges are perfect.
However, since wireless communication is inherently unreliable (e.g., due to fading and interference, etc.), the message transmissions containing the information about evaluating the objective may not always be successful, resulting in failure of acquisition of the information at the intended time of operating the algorithm.
Our numerical evaluation reveals that a straightforward solution using SA to circumvent this problem performs very poorly in terms of its computational speed, and thus appears to be far from being practical in a situation where the message drop rate is high.
The main purpose of this paper is to develop an efficient way of implementing the SA algorithm for wireless networks even under a physical channel that experiences frequent message drops.

The main contributions of this paper are as follows.
\begin{enumerate}
\item We investigate an important performance issue that arises from the unreliable nature of wireless communications in implementing the SA algorithm in general distributed wireless networks.
\item We propose a new algorithmic approach that can deal more efficiently with an impact from the imperfect communications, and rigorously prove the \emph{optimality} of the proposed algorithm under the standard SA framework.
\item We demonstrate that the proposed algorithm offers significant improvement in terms of its computational speed in networks with high message drop rates.
\end{enumerate}

We organize this paper as follows. First, we provide a brief overview of related work in Section \Romannum{2}, and some preliminaries in Section \Romannum{3}.
In Section \Romannum{4}, we describe the detailed implementation structure of an algorithm that is based on the SA approach, and describe the main problem we focus on in this paper.
In Section \Romannum{5}, we present our new idea to deal with the problem, along with a mathematical analysis for the optimality and efficiency of our solution. 
Section \Romannum{6} provides numerical evaluations that support our main arguments. In Section \Romannum{7}, we discuss some practical considerations and conclude the paper in Section \Romannum{8}.

%% file: 2_related.tex
\section{Related Work}

The scope of this paper is closely related to the problem of designing wireless link scheduling algorithms.
In particular, the message issue we have introduced in the previous section is of importance in the development of recently studied CSMA-type distributed algorithms \cite{NTS10QCSMA, JW10DC, SRS09}.
We provide a brief overview of previous works, and emphasize again the significance of our contributions in this context.

Recently, a suite of CSMA-type algorithms have gained a lot of attension in the research community.
These algorithms are known to be throughput optimal and can be easily implemented in a distributed manner requiring minimal message overheads. 
The key enabler of this success is the utilization of an SA-like algorithm to solve the max-weight problem.
While the goal of achieving throughput optimality is to generate a sequence of schedules such that the long-term service rates can support any feasible arrival rates, the task of solving max-weight problem plays a critical role in this job and it can be indeed leveraged to achieve optimality.

We should point out that the type of messaging used in these algorithms depends on which capacity model is used in their problem setting.
In earlier works \cite{SRS09, JRJ10CDC, JW10DC, NTS10QCSMA}, the algorithms are typically developed under a simple binary capacity model:
each link can be either active or inactive, where activation of two links at a close distance leads to collision, i.e., both transmissions fail. 
In this model, there is few restrictions on the way in which the necessary information is collected.
This is because the only information needed to decide whether to activate a link's transmission is to know whether any of its neighboring links (the set of links that interfere with it) is active. This can be easily done by having each active neighboring link send a one-bit signal\footnote{Some overheads such as headers and/or guard times may be necessary depending on the types of practical systems, as in \cite{NTS10QCSMA}.} on a predefined and commonly shared time slot in order to convey its activity state, and the link simply detects the presence of the combined signal.

However, this information acquisition scheme may not be able to be used on other more realistic capacity models. One such an example is the Signal-to-Noise Ratio (SINR) model, in which links obtain capacity proportional to the ratio of their signal strength to the interference experienced in their receiver. The reason is that in this model there may not exist a clear condition that distinguishes between collision and not collision, but the degree of capacity degradation caused by the activation of other links may be different for different links depending on their transmission power mode, geographical distance between them, and etc. In this case, more detailed information about the capacity degredation from different nodes may have to be collected individually. Indeed, this increased message complexity can be a critical source of the incomplete message acquisition problem as we will explain later.

There are a few works that have extended the CSMA algortihms to the SINR model case \cite{swamy2017adaptive, SINR13TON}. 
However, these works are restricted to the use of a \emph{threshold-type} capacity model, i.e., a link obtains a unit capacity if its SINR is above a certain threshold, and zero otherwise. 
This condition is a critical assumption that allows to use the above mentioned information acquisition scheme and avoids the message complexity problem. This capacity model, however, does not allow the wireless nodes to use adaptive modulation and coding techniques to increase data rates for higher SINR.

In \cite{borst2011distributed} and \cite{qian2010globally}, the authors have considered general capacity models, not restricted to threshold-type ones. 
However, they ignore the message complexity issue, and assumed that all suitably defined local information needed to perform their algorithms is readily available at the time of operating the algorithm. In this paper, we do not assume such an oracle, but explicitly consider the impact due to imperfect collection of required information, and develop a solution to the problem.

It is also worthwhile to mention that the delay performance of these algorithms in queueing systems is highly affected by their computational speed. According to the standard queueing theory \cite{BigQueues, Bremaud03}, the correlation on arrival and/or service processes has an adverse impact on the queueing delay. 
As we will show later, our new solution is very helpful in improving the algorithm operation speed, which in turn generates more rapidly evolving and less correlated link service processes in comparison to a naive approach.
From this view point, the significance of the aforementioned contributions in the scheduling problem can be translated into the fact that our proposed algorithm, applied to max-weight type problems with any general capacity model, guarantees throughput-optimality while reducing delay performance degradation due to the imperfect communications.
However, aside from this significant merit, achieving faster computational speed to generate an equivalent solution is evidently desirable in designing this type of randomized algorithms for many applications.

%% file: 3_prelim.tex
\section{Preliminaries}


\subsection{System model and objective}

\textbf{Network model.} 
We consider a wireless network consisting of a set $\N$ of $n$ communication links (transmitter-receiver pairs). 
Each link-$i$ transmitter node has its local parameter $x_i$ that determines its transmission power level from a discrete set $\{0,\ldots,P^{\max}\} \triangleq \M$. Let $\x = \{x_1,\ldots, x_n\}$.
Links interfere with each other such that a transmission of one link is treated as interference at other links. 
We consider the SINR-based interference model. 
That is, when each link $i$ sends a signal with its power level $x_i$, the receiver of each link $i$ attains its SINR level, $\gamma_i(\x) = \frac{g_{ii}x_i}{\sum_{j \neq i} g_{ji}x_j + n_0}$, where $g_{ij}$ is the channel gain from link-$i$ transmitter to link-$j$ receiver, and $n_0$ is the thermal noise. The link $i$ then obtains its transmission rate $c_i(\gamma_i(\x))$ as a function of the experienced SINR level, which is a typically monotone function, such as $\log(1+\gamma_i(\x))$.
Let $\bX \triangleq \M^n$, and call an instance $\x = \{x_1,\ldots, x_n\} \in \bX$ \emph{configuration}. We denote by $c(\x) = \{c_i(\x)\}_{i \in \N}$ a capacity vector with configuration $\x$. We will also use $\x_{[S]} = \{x_i\}_{i \in S}$, for $S \subseteq \N$ to denote a subset $S$ of configuration $\x$.


\textbf{Main objective.}
We require that each link controls its transmission power level in order to achieve a certain performance objective.
Specifically, we aim at designing a distributed algorithm that makes decisions $\x(t) \in \bX$ so that the long-term time proportion of the configuration converges to a solution to
\begin{eqnarray*}
\mbox{(\textbf{OPT-MW})} \quad \text{maximize}_{\x \in \bX} \quad
\sum_{i \in \N} w_i c_i(\x),
\end{eqnarray*}
where $w_i$ is a weight of link $i$. In the following, we call the pair of product $w_ic_i(\x) \triangleq f_i(\x)$ \emph{performance objective}, $f_i:\bX \rightarrow \R$, of link $i$ associated with each configuration $\x \in \bX$.


A significant motivation for considering \textbf{OPT-MW} is its relevance to the throughput-optimality in queueing systems.
To be more specific, suppose that each link maintains a queue fed by an exogenous packet arrival process.
In \cite{neely2003power}, it was shown that if in each time slot, a configuration is selected according to the above max-weight rule, where the weight is queue size, then the queues can be stabilized (keeping all link queues finite) for all arrival vectors that are within the capacity region determined by the convex combination of capacity vectors with all possible configurations.
While in our problem setting the weight parameters are assumed to be constant, an algorithm that solves \textbf{OPT-MW} with large enough weights can be shown to be throughput-optimal based on the time-scale separation assumption, and the assumption can be relaxed by the recent queue-based adaptation schemes \cite{SRS09, JRJ10CDC}.



In general, \textbf{OPT-MW} is known to be an NP-hard problem, and therefore it is unlikely that there exists an efficient algorithm to solve it even in a centralized manner.
Our solution approach to the problem is to utilize the \emph{simulated annealing} method, which is known to guarantee to find the optimal solution with high probability in a certain asymptotic sense even for NP-hard problems.



\subsection{Simulated annealing}

Central to the idea of simulated annealing is the Metropolis Hastings (MH) algorithm, which is a Monte Calro Markov Chain (MCMC) method that can be used for obtaining a sequence of samples from a given probability distribution. 
We here briefly review the MH algorithm and its relation to simulated annealing to solve \textbf{OPT-MW}.

Consider an irreducible Markov chain $X_t$ with a finite state space $\Omega$ and its transition probability matrix $\p = \{P_{ij}\}_{i,j \in \Omega}$.
Let $\PI = \{\pi\}_{i \in \Omega}$ be a probability distribution over the state space.
The MH algorithm is intended to obtain a transition probability matrix $\p$ that has $\pi$ as its stationary distribution while satisfying the reversibility condition, i.e., $\pi_i P_{ij} = \pi_j P_{ji}$.
The details of the MH algorithm are described as follows. At the current state $i$ of $X_t$, the next state $X_{t+1}$ is proposed with a probability with \emph{proposal} distribution $c_{ij}$ - the state transition probability of an arbitrary irreducible Markov chain on the same state space, where $c_{ij} > 0$ if and only if $c_{ji} > 0$. The proposed state transition is accepted with probability $\alpha_{ij} = \text{min}\left\{1, \frac{\pi_j c_{ji}}{\pi_i c_{ij}} \right\}$, and is rejected with probability $1-\alpha_{ij}$. Therefore, the transition probability $P_{ij}$ is given by $P_{ij} = c_{ij} \alpha_{ij} = \text{min}\{ c_{ij}, c_{ji} \pi_j / \pi_i \}$,
for $i\neq j$, and $P_{ii} = 1 - \sum_{j \neq i } P_{ij}$. 
When the proposal distribution is symmetric, i.e., $c_{ij} = c_{ji}$ for all $i,j \in \Omega$, the form of transition probabilities reduces to $P_{ij} = c_{ij} \min\{1, \pi_j / \pi_i\}$.

As an application of the MH algorithm, an important class of probability distribution to be used for solving combinatorial optimization problems is \emph{Boltzmann-Gibbs distribution}, which is typically constructed for \textbf{OPT-MW} by
\begin{equation} \label{eq:stationary}
    \pi(\x) = \frac{1}{Z} e^{\beta f(\x)}, \quad \x \in \bX,
\end{equation}
where $f(\x) = \sum_{i \in \N} f_i(\x)$, $Z$ is the normalization constant: $Z = \sum_{\x' \in \bX} e^{\beta f(\x')}$, and $\beta > 0$ is a parameter related to capturing the trade-off between optimality and convergence speed. Evidently, as $\beta$ becomes large, the probability distribution will be concentrated on the set of optimal solutions $\bX^* := \{ \x \in \bX : f(\x) = \max_{\x' \in \bX} f(\x') \}$. 
In this form of $\PI = \{\pi(\x)\}_{\x \in \bX}$, the constructed transition probabilities by the MH algorithm can be written as $P(\x,\x') = c(\x,\x') e^{-\beta [f(\x) - f(\x')]^+}$, for $\x,\x' \in \bX$ $(\x \neq \x')$ and $P(\x,\x) = 1 - \sum_{\x' \in \bX} P(\x,\x')$, given that a suitably defined proposal distribution $c(\x,\x')$ is symmetric.

The simulated annealing is an adapted version of the MH algorithm. 
The most distinct feature of SA is that it allows $\beta$ to increase monotonically in time, but with sufficiently slowly varying rate, in order to guarantee the convergence to the optimal solution in a certain probabilistic sense. 
The time-varying parameter $T(t) = 1/\beta(t)$ is often referred to as the \emph{temperature} at time $t$, and the sequence of $T(t)$ is called \emph{cooling schedule}.
Many proofs of convergence of cooling schedules have already appeared in the literature~\cite{connors1989simulated,hajek1988cooling}. We defer discussion of this topic in Section \ref{sec:rsa}-B.

%% file: 5_challenge.tex
\section{The Implementation and The Challenge}

\subsection{Implementation structure}

\algnewcommand\algorithmicpick{\textbf{Pick phase:}}
\algnewcommand\PICK{\item[\algorithmicpick]}
\algnewcommand\algorithmictrain{\textbf{Train phase:}}
\algnewcommand\TRAIN{\item[\algorithmictrain]}
\algnewcommand\algorithmicmessaging{\textbf{Messaging phase:}}
\algnewcommand\MESSAGING{\item[\algorithmicmessaging]}
\algnewcommand\algorithmicdecision{\textbf{Decision phase:}}
\algnewcommand\DECISION{\item[\algorithmicdecision]}
\algnewcommand\algorithmicdecisionati{\textbf{Decision phase:} (at node $i$)}
\algnewcommand\DECISIONatI{\item[\algorithmicdecisionati]}

\algnewcommand\algorithmicinput{\textbf{Input:}}
\algnewcommand\INPUT{\item[\algorithmicinput]}

\begin{algorithm}[t] 
    \caption{Basic SA (BSA) Algorithm (in time slot $t$)} \label{alg:basic}
    \begin{algorithmic}[1]
    \PICK 
    \State The network selects a link $i \in \N$ u.a.r.
    \State The link $i$ chooses $x_i(t) \in \M \backslash \{x_i(t-1)\}$ u.a.r.
    \State Set $x_j(t) = x_j(t-1)$, $\forall j \in \N \backslash \{i\}$.
    \TRAIN 
    \State Test the new configuration $\x(t)$.
    \State Every link $j \in N$ locally measures $f_j(\x(t))$.
    \State Set $\Delta_j = f_j(\x(t)) - f_j(\x(t-1))$.
    \MESSAGING
    \State Each link $j \in N_i$ sends $\Delta_j$ to link $i$.
    \DECISIONatI
    \State{Set $\Delta = \Delta_i + \sum_{j \in N_i} \Delta_j$. } 
    \State{\algorithmicif \; $\Delta \leq 0 $ \; \algorithmicthen \; $x_i(t) = x_i(t-1)$ w.p. $1-e^{\beta \Delta}$ }
    \end{algorithmic} 
\end{algorithm}

Realizing the SA idea in a distributed network requires a considerable attention since the specific implementation in practical networks will differ greatly depending on the characteristics and the constraints of the network systems.
We present our implementation structure of the SA idea to be performed in the SINR model.

In our implementation, time is divided into discrete time slots where each time slot $t$ consists of four phases which include pick, training, messaging, and decision. 
In the pick phase, the network selects a link $i \in \N$ uniformly at random (u.a.r.).
The task of selecting a random link can be done in a distributed manner by having each link trigger an independent poisson clock with a unit rate over continuous time domain, and by suitably defining a time slot as an interval of each clock tick.
The selected link generates a new power level state $x_i(t)$ u.a.r. different from its previous state $x_i(t-1)$.
The newly generated configuration $\x(t)$ is then tested by having each transmitter node transmit a test signal with the selected power level, and the receiver node of each link measures its performance objective. Each receiver node then constructs a message containing the objective differential - the measured quantity subtracted by that of previous time slot - and transmits it to the transmitter node of link $i$ during the messaging phase.
Upon receiving the messages, the link $i$ decides whether to accept the new power level state or not, based on the received information and the previously described MH algorithm to achieve $\PI$ in Eq (\ref{eq:stationary}).

Note that if links are located sparsely over a geographical region and the channel gain quickly decreases with the distance between a receiver and an interfering transmitter, then it is reasonable to assume that the interference from links that are far away can be ignored. 
Specifically, we define a neighbor set $N_i$ for each link $i$ such that a link $j$ belongs to the neighbor set $N_i$ if link-$j$ receiver is located within a given radius of link-$i$ transmitter, and will consider only those links in the neighbor set as the primary sources of the interference. 
Therefore, each link-$i$ transmitter node only needs to collect information from its neighboring links $j \in N_i$ during the messaging phase.
The detailed algorithm is outlined in Algorithm~\ref{alg:basic}.

At first glance, our implementation appears to be similar to the standard PICK-and-COMPARE methods as introduced in \cite{Eryilmaz2006, hwlee2012, Modiano2006Gossip}. The main idea of the previous approaches is to have \emph{every} node generate its new random power level, and compare its objective value with that of the previous power allocation. If the new power allocation improves the objective value, then the new allocation is accepted to use in the next time slot, and if otherwise, remains to use the previous one. However, in multi-hop wireless networks, this comparison task is very challenging because it requires to compare the network-wide weighted-sum rates achieved by the two power allocation. To this end, they adopted a gossip-like algorithm, however, the computation of each power allocation using the gossip algorithm requires up to $O(n^3)$ information exchange, which may not be easily implementable for large networks. On the other hand, we do not require such a network scale comparison, as we perform the comparison task at link level. That is, we propose to change only a single state at a time, which makes the computation of the objective differential easy and suitable to be implementable in a distributed manner. 

\subsection{The challenge with imperfect communications}

We have described the basic SA algorithm based on the assumption that the message containing the objective differential locally measured at each node is perfectly delivered to the intended node during the messaging phase. In practice, however, the delivery of messages may not always be successful, and there can be several reasons that can prevent the message delivery from being successful. 

1) \emph{Fading.} A primary reason for the delivery failure is due to the inherently unreliable nature of wireless communications. 
In wireless communications, the transmission channel suffers from temporal variations in its condition with various variables, and this can often result in a great amount of signal attenuation and message decoding errors.

2) \emph{Message complexity.}
When the network experiences frequent events of join and leave of nodes, it may not be easy for each node to find a proper coordination in a deterministic way for receiving multiple messages from different neighbors. To deal with such a potential dynamics, an \emph{Aloha}-type of randomized neighbor discovery method can be used as an alternative, e.g. \cite{vasudevan2013efficient}.
One way to do is to allocate multiple sub-slots during the messaging phase, and in each sub-slot, nodes transmit their message with some probability. 
In this way, nodes can deliver messages while avoiding collisions in a randomized fashion. However, there is always a chance that the delivery of messages may not be successful, since the number of sub-slots is finite and fixed.

We capture various factors that can cause message drops by means of probability to represent the combined effect.
In specific, we assume that in each time $t$, the selected node $i$ is only able to collect a subset $S(t)$ of nodes from its neighbors $N_i$ with some unknown probability $q_{i,S(t)}$, $S(t) \subseteq N_i$, which is $i.i.d.$ over time slots, where the probability of collecting the full set information is assumed to be non-zero, i.e., $q_{i,N_i} > 0$.

This limited capability of the message passing poses the following practical challenge: when the subset $S(t)$ of information collected at time $t$ is strictly smaller than $N_i$, the node $i$ cannot compute the state transition probability correctly, and therefore it is unclear how to behave in this time slot. 
A straightforward idea to deal with the problem is as follows.
If the intended node gathers all the information from the full set of its neighbors successfully, then the node performs Algorithm 1. And, if otherwise, it defers performing the algorithm and simply maintains the current state (Algorithm \ref{alg:naive}).

\algnewcommand\algorithmicmsginput{\textbf{Message Input :}}
\algnewcommand\MSGINPUT{\item[\algorithmicmsginput]}

\begin{algorithm} 
    \caption{Lazy SA (LSA) algorithm (at node $i$ in time $t$)} \label{alg:naive}
    \begin{algorithmic}[1]
    \MSGINPUT $S(t) \subseteq N_i$ \text{and} $\{\Delta_j\}_{j \in S(t)}$.
    \DECISION
    \State{\algorithmicif \; $S(t) \equiv N_i$ \algorithmicthen \; perform (8-9) in Algorithm 1} 
    \State{\algorithmicelse \; $x_i(t) = x_i(t-1)$ }
    \end{algorithmic}
\end{algorithm}

We verify that this algorithm has its stationary distribution as $\PI$ in Eq. (\ref{eq:stationary}), of which proof is in Appendix.

\begin{proposition} \label{lem:LSA_equiv} \! The stationary distribution of LSA algorithm is $\PI$.
\end{proposition}

As one can notice, the problem of this algorithm is its slow computational speed. Suppose that a node has multiple neighbors and the messages each from different neighbors drop independently with some non-zero probability. Then, the probability that it obtains all the information so that it can perform the algorithm decreases exponentially fast with the number of neighboring nodes. 
Next, we present a new approach that can greatly improve the algorithm operation speed in the presence of message drops.

%% file: 6_algorithm.tex
\section{Improving the computational speed} \label{sec:rsa}

\subsection{Proposed solution: rapid SA (RSA) algorithm}

The high level description of the main idea we introduce here is as follows.
In many network application scenarios, a change of a single nodal configuration often results in a limited amount of impact to the dependent performance objectives. 
With the knowledge of the bounded impact, we construct a confidence range on the objective differential that can be made due to the change of configuration, and utilize it to compute the desired level of probabilistic uncertainty in the stochastic acceptance criterion of SA.
As a result, a certain level of impreciseness on the evaluation of objective differentials can be tolerated without affecting its optimality.

To give a motivating example, consider the following simple network scenario with a set of four nodes, $\N = \{a,b,c,d\}$, where each node represents a distinct pair of communication link.
Each node $i \in \N$ can be either active, $x_i = 1$, or inactive, $x_i = 0$, and two nodes connected in the graph (presented in Fig \ref{fig:example_tp1}) \emph{conflict} with each other such that a node obtains a unit capacity only if it is active and all its neighbor nodes (the set of nodes connected to it in the graph) are inactive, and obtains zero capacity if otherwise. 
The objective is to maximize $f(\x) := \sum_{i \in \N} w_i c_i(\x)$ where the weights $w_i$'s are chosen as $w_a = 5, w_b = 7, w_c = 10, w_d = 3$.
With this setup, suppose that the configuration at current time $t$ is $\x(t) = \{1,1,0,0\}$, i.e., only $a$ and $b$ are active, and consider to switch the state of node $c$ from inactive to active, i.e., $\x(t+1) = \{1,1,1,0\}$ according to the SA framework. In this case, the local objective differential measured (during the train phase) at each node is $\Delta_a = -5$, $\Delta_b = -7$, $\Delta_c = 0,$ and $\Delta_d = 0$, respectively, and the values $\Delta_a$, $\Delta_b$, $\Delta_d$ are to be transmitted in the messaging phase towards node $c$. Suppose further in this particular time slot, $\Delta_a$ and $\Delta_b$ are delivered successfully whereas $\Delta_d$ has not reached node $c$ due to a temporally bad condition experienced over the communication channel. 
Since node $c$ did not receive $\Delta_d$, it cannot correctly compute the aggregate objective differential which is needed to compute the transition probability. 
On the other hand, with the knowledge of $\Delta_a$ and $\Delta_b$, node $c$ can determine a bounded range on the consequential aggregate objective differential such that $\Delta_a + \Delta_b + \Delta_c + \Delta_d = \Delta \in [-15, -12]$, since $\Delta_d \in \{-3, 0\}$ can be easily inferred by node $c$: $\Delta_d = -3$ if node $d$ was active, and $\Delta_d = 0$ if it was inactive. Our main idea we propose here is to suggest to make a transition based on the lower bounded transition probability ($e^{- 15 \beta }$ in this case, rather than $e^{-12 \beta }$, that with the true objective differential) that can be computed based on any subset information. 

This new idea relies on the following assumption:
each node $i$ has a known lower bound (upper bound in minimization problem) on the differential contribution to the objective $f_j$ of any neighboring node $j$ that can be made due to solitary change of node $i$'s configuration from $x_i$ to $x'_i$ such that
\begin{equation*}
    \min_{\x_{[-i]}}f_j(x'_i, \x_{[-i]}) - \max_{\x_{[-i]}} f_j(x_i, \x_{[-i]}) \geq b^{ij}_{x_i,x'_i},
\end{equation*}
where $\x_{[-i]} = \x_{[\N \backslash \{i\}]}$, and it is allowed to have $b^{ij}_{x_i,x'_i}= -\infty$ in the case that there is no known bound for it.
For the max-weight problem under the SINR model, one can obtain a trivial bound: $b^{ij}_{x_i,x'_i}$ is $-w_j c_j^{\max}$ if $x'_i \geq x_i$ and is zero if otherwise, where $c^{\max}_j$ is the (a priori known) maximum achievable rate of link $j$ due to physical constraints of wireless technology in use, and $w_j$ can be easily informed as it only requires a one-time transmission. It is possible to obtain a tighter bound if additional information on the objective function, such as the gain term $g_{ij}$ between node $i$ and $j$, is available. The efficiency of this approach essentially relies on the tightness of the bounds, however, we observe through extensive simulations that loose bounds are often sufficient to offer substantial improvement on the algorithm operation speed when the packet drop rate is high.
A formal description of this idea is presented in Algorithm~\ref{alg:new}.


\begin{figure}[t]
    \centering
    \includegraphics[scale=0.45]{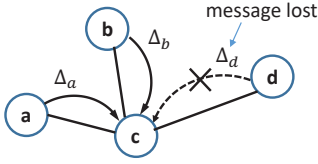}
    \caption{An example topology in which the messages $\Delta_a$ and $\Delta_b$ are delivered successfully to node $c$, whereas the massage $\Delta_d$ is lost.} \label{fig:example_tp1}
\end{figure}

\begin{algorithm} 
    \caption{Rapid SA (RSA) algorithm (at node $i$ in time $t$)} \label{alg:new}
    \begin{algorithmic}[1]
    \MSGINPUT $S(t) \subseteq N_i$ \text{and} $\{\Delta_j\}_{j \in S(t)}$.
    \DECISION 
    \State{Set \resizebox{.42 \textwidth}{!} {$\Delta_{[S(t)]} = \Delta_i \!+\! \sum_{j \in S(t)}\! \Delta_j \!+\! \sum_{j \in \N_i \backslash S(t)} b^{ij}_{x_i(t-1),x_i(t)} $}.}
    \If{ $\Delta_{[S(t)]} \leq 0 $} $x_i(t) = x_i(t-1)$ w.p. $1-e^{\beta \Delta_{[S(t)]}}$
    \EndIf
    \end{algorithmic}
\end{algorithm}

In RSA algorithm, nodes are allowed to perform the algorithm based on the bounded estimate on the potential objective differential, which can be computed based on the subset information currently observed.
Compared to the LSA algorithm, we add additional transitions on the system dynamics, so its faster computational speed is expected. In the previous case, for example, when the message $\Delta_d$ was lost, the network configuration had to remain on the same state in LSA algorithm, whereas now it has some degree of probability that can transit to a new state in RSA algorithm.
We obtain the following relation among the algorithms, of which proof is provided in Appendix.



\begin{proposition} \label{prop:ordering}
Let $\bP^{B}$, $\bP^{R}$, and $\bP^{L}$ denote the transition probability matrices of BSA, RSA, and LSA algorithms, respectively. Then, for all $\x, \x' \in \bX$ $(\x \neq \x')$, it holds
\begin{equation*}
    P^B(\x,\x') \geq P^R(\x,\x') \geq P^L(\x,\x').
\end{equation*}
\end{proposition}



\input{6_algorithm_optimal}

%% file: 6_algorithm_optimal.tex
\subsection{Optimality}

Note that RSA algorithm does not necessarily achieve the same stationary distribution $\PI$ we intended, and it is difficult to find a closed form solution for it.
Technically speaking, the algorithm experiences \emph{bias} on the desired stationary distribution due to the additional transitions we added onto the algorithm. 
For this reason, the conceptual argument that the probability distribution gets concentrated on the optimal states as $\beta$ grows cannot be used.
Our main concern here is therefore to see if the algorithm is still able to find optimal solutions under the standard SA framework.
To that end, we first formally define the notion of an algorithm being optimal.
\begin{definition}
An algorithm is called \emph{annealing-optimal} if a Markov chain, $X(t)$, governed by the algorithm with a proper cooling schedule for $\beta(t)$ achieves
\begin{align}
    \lim_{T \rightarrow \infty} \frac{1}{T} \sum_{t=1}^T \pr\{X(t) \in \bX^*\} = 1. \label{eq:anneal_opt}
\end{align} 
\end{definition}

In the conventional SA, the cooling schedule is typically constructed by $\beta(t) = \log(t)/d$ where $d$ is some positive constant that determines the order of cooling rate. 
Using this cooling schedule, by the proper cooling schedule we mean that an algorithm is said to be annealing optimal if Eq. (\ref{eq:anneal_opt}) can be verified for sufficiently large enough $d$.

To verify the optimality of RSA algorithm, we adopt a technical method introduced in \cite{connors1989simulated}, in which the optimality of the original simulated annealing algorithm is proven.
The authors in \cite{connors1989simulated} have verified the annealing optimal of SA algorithms for a certain class of Markov chains whose transition probabilities can be written as
\begin{equation}
    p_{ij}(t) = c_{ij} \epsilon(t)^{V_{ij}}, \label{eq:trans_pr}
\end{equation}
where $V_{ij}, c_{ij} \geq 0,$ for all $i,j$, $\sum_{j \neq i} c_{ij} = 1$, for all $i$, $p_{ii}(t) = 1 - \sum_{j \neq i} p_{ij}(t)$, and $0 \leq \epsilon(t) \leq 1$, $t \geq 1$ is the parameter related to the cooling schedule. 
Note that the conventional simulated annealing algorithm can be represented by this form with setting $V_{ij} = [f(j) - f(i)]^+$ and $\epsilon(t) = e^{-\beta(t)}$ in which minimum $f(\cdot)$ is sought.
It is a straightforward job to verify that both BSA and LSA algorithms can be represented by the above form, from which their optimalities easily follow.

On the other hand, it turns out that the transition probabilities of the RSA algorithm does not conform to Eq. (\ref{eq:trans_pr}), and thus their analysis cannot be immediately applied to show its optimality. Nevertheless, we obtain the following result.



\begin{theorem} \label{thm:optimal}
RSA algorithm is annealing optimal.
\end{theorem}

The major part of the analysis is to generalize the transition probability form of Eq. (\ref{eq:trans_pr}) in order to represent multiple conditional transition probabilities of RSA algorithm for different message acquisition events, and to verify a suitably defined notion of \emph{recurrence order} of each state, which conceptually captures how likely the system tends to stay on the state in a certain asymptotic sense, remains the same as that of BSA algorithm albeit the generalization. 
For brevity of the presentation, we provide the detailed steps for the proof in Appendix.

\subsection{Efficiency in asymptotic variance rate}

We now provide an insight into understanding the benefit of the proposed approaches by comparing different algorithms: BSA, LSA, and RSA algorithms. 
To quantitatively analyze and compare these algorithms, we first need to choose a specific metric that characterizes one algorithm being a good one.

One popular metric often considered in the literature is the mixing time. 
Conceptually, the mixing time of a Markov chain is the time until the Markov chain is close to its stationary state.
In the standard Markov chain theory \cite{Levin09}, the mixing time is precisely defined as
\begin{equation*} 
\resizebox{.5 \textwidth}{!} {
    $t_{mix}(\epsilon') = \min\{t \geq 1 : \max_{i \in \Omega} \| P_t(i,A) - \pi(A) \|_{\text{TV}} \leq \epsilon', \forall A \subseteq \Omega \},$
}
\end{equation*} 
which is the formalization of the idea: how large must $t$ be until the time-$t$ distribution is $\epsilon'$-close to $\PI$.
Unfortunately, directly dealing with this quantity is a very difficult task, and most of existing analytic techniques rely on the spectral analysis based on the relation $t_{mix}(\epsilon') \leq \log(1/(\epsilon' \pi_{\min})) / (1-\text{SLEM}(\bP))$, where $\text{SLEM}(\bP)= \max\{\eta_2, |\eta_{|\Omega|}|\}$ is the second largest eigenvalue modulus and $1 = \eta_1 \geq \ldots \geq \eta_{|\Omega|} \geq -1$ are the left eigenvalues of $\bP$.
Although the common wisdom in the literature is that the smaller SLEM is the smaller mixing time the chain $\bP$ will have, its ordering relation on the upper bounds does not necessarily imply the chain with a smaller SLEM will actually mix faster in a rigorous sense.

Instead, we look at another performance metric that has been extensively used in the sampling theory.
Sampling schemes are often used to estimate $\Ex_{\PI}(h) \triangleq \sum_{i \in \Omega} h(i) \pi(i)$ for various functionals $h:\Omega \rightarrow \R$ by generating $t$ samples $\{X(s)\}_{s=1}^t$ and constructing an estimator $\hat{\mu}_t(h) = \frac{1}{t} \sum_{s = 1}^{t} h(X(s))$.
In assessing the accuracy of this estimator, the asymptotic variance rate has been used as an important criterion in the literature.
The asymptotic variance rate $\sigma(\bP,h)$ of the estimate $\hat{\mu}_t(h)$ is defined in \cite{Levin09} as
\begin{align}
    \sigma(\bP,h) & = \lim_{t \rightarrow \infty} t \cdot \text{Var}(\hat{\mu}_t(h)). \label{eq:avr}  
\end{align}

It has been known that the quantity $\sqrt{t} ( \hat{\mu}_t(h) - \Ex_{\PI}(h) )$ converges in distribution to a Gaussian random variable with zero mean and variance $\sigma(\bP,h)$. 


We consider the ordering relationship among the three algorithms in term of the asymptotic variance rate. A useful technique related to this task is the so-called \emph{Peskun ordering}, which is described next.

\begin{definition}[Peskun ordering] \cite{peskun1973optimum}
 For two finite irreducible Markov chains on a finite state space $\Omega$ with $\bP=\{P_{ij}\}_{i,j \in \Omega}$ and $\bP'=\{P'_{ij}\}_{i,j \in \Omega}$ with the same stationary distribution $\PI$, it is said that $\bP'$ dominates $\bP$ off the diagonal, written as $\bP \preceq \bP'$ if $P_{ij} \leq P'_{ij}$ for all $i,j \in \Omega$ ($i \neq j$).
\end{definition}

\begin{lemma} \cite{peskun1973optimum}
If $\bP$ and $\bP'$ are reversible with respect to $\PI$, and $\bP \preceq \bP'$, then $\sigma(\bP, h) \geq \sigma(\bP',h)$ for any $h$ with $\text{Var}_{\PI}(h) \triangleq \sum_{i \in \Omega} (h(i) \pi(i) - \Ex_{\PI}(h))^2 < \infty$.
\end{lemma}

Note from Proposition \ref{lem:LSA_equiv} that the stationary distributions of LSA and BSA algorithms are identical as $\PI$ in Eq. (\ref{eq:stationary}).
Also, the relation $\bP^{B}(\x,\x') \geq \bP^{L}(\x,\x')$ for $\x,\x' \in \bX$ $(\x \neq \x')$ in Proposition \ref{prop:ordering} is exactly the definition of Peskun ordering. Therefore, we obtain the following consequence.

\begin{proposition} $\sigma(\bP^L, h) \geq \sigma(\bP^B, h)$, for any $h$ with $\text{Var}_{\PI}(h) < \infty$.
\end{proposition}

However, the stationary distribution of RSA algorithm, denoted by $\PI^{R}$, is not necessarily equivalent to $\PI$. For this reason, we cannot rely on the Peskun ordering relation between RSA and the others. Unfortunately, the efficiency analysis for comparing two Markov chains with different stationary distributions is notoriously difficult, and to the best of our knowledge there are no known technical tools applicable to our case. We leave the following statement as our conjecture.

\begin{conjecture} \label{conj:ordering} $\sigma(\bP^L,h) \geq \sigma(\bP^R,h) \geq  \sigma(\bP^B,h)$, for any $h$ with $\text{Var}_{\PI}(h)< \infty$ and $\text{Var}_{\PI^{R}}(h) < \infty$.
\end{conjecture}

The rationale for the conjecture is that
we observed from various simulations that the distributional bias of RSA algorithm is often very small, and hence we expect from Proposition \ref{prop:ordering} that a similar ordering relation will hold.
This conjecture is empirically found to be true in diverse cases.

%% file: 7_simulation.tex
\section{Numerical Evaluation}

In this section, we present the numerical experiments for the proposed algorithms.
We first consider the network scenario of the four link case presented in Section \ref{sec:rsa},
where simulations are performed with using fixed but different temperature parameters in order to observe how different algorithms behave in a specific temperature regime. 
We assume that the weight parameters, $\{w_i\}$, are given and fixed as such described in the earlier section, and each node knows these parameters for all of its neighbor nodes. And, we chose $b^{ij}_{x_i,x'_i}$ is $0$ for $x_i = 1, x'_i = 0$, and is $-w_j$ for $x_i=0, x'_i=1$, for all $i$ and $j \in N_i$ for the bound parameters.
Fig \ref{fig:st_dist}-A and \ref{fig:st_dist}-B plot the stationary distributions obtained from different algorithms for $\beta = 0.1$ and $\beta = 1$, respectively. Messages generated from neighbor nodes are set to be lost independently with probability $0.5$ for LSA and RSA algorithms. 
Note that the results of BSA algorithm corresponds to those of LSA (or RSA) algorithm with no message loss events.
As expected from the traditional analysis of SA and the form of Gibbs distribution, it can be seen from the figure that the distribution from BSA algorithm is scattered around different states for small $\beta$ ($\beta=0.1$), whereas it becomes concentrated on the optimal state, $(1,1,0,1)$, for large $\beta$ ($\beta=1$). 
Similarly, the results of both LSA and RSA algorithms also match well with them, which reveals that the same annealing optimality will hold for RSA algorithm as well.

\begin{figure}[t]  \vspace{-4mm}
  \subfloat[$\beta=0.1$]{\includegraphics[width=1.75in]{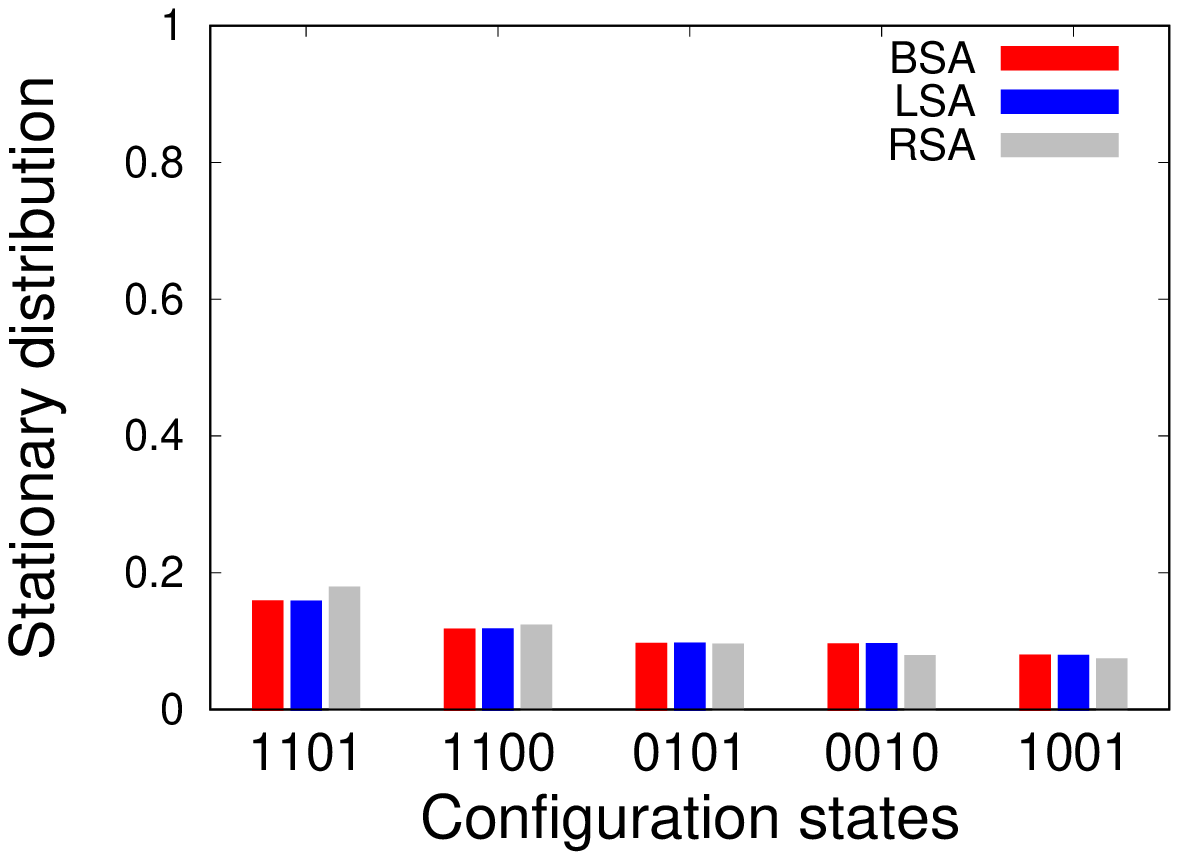} }
  \subfloat[$\beta=1$]{\includegraphics[width=1.75in]{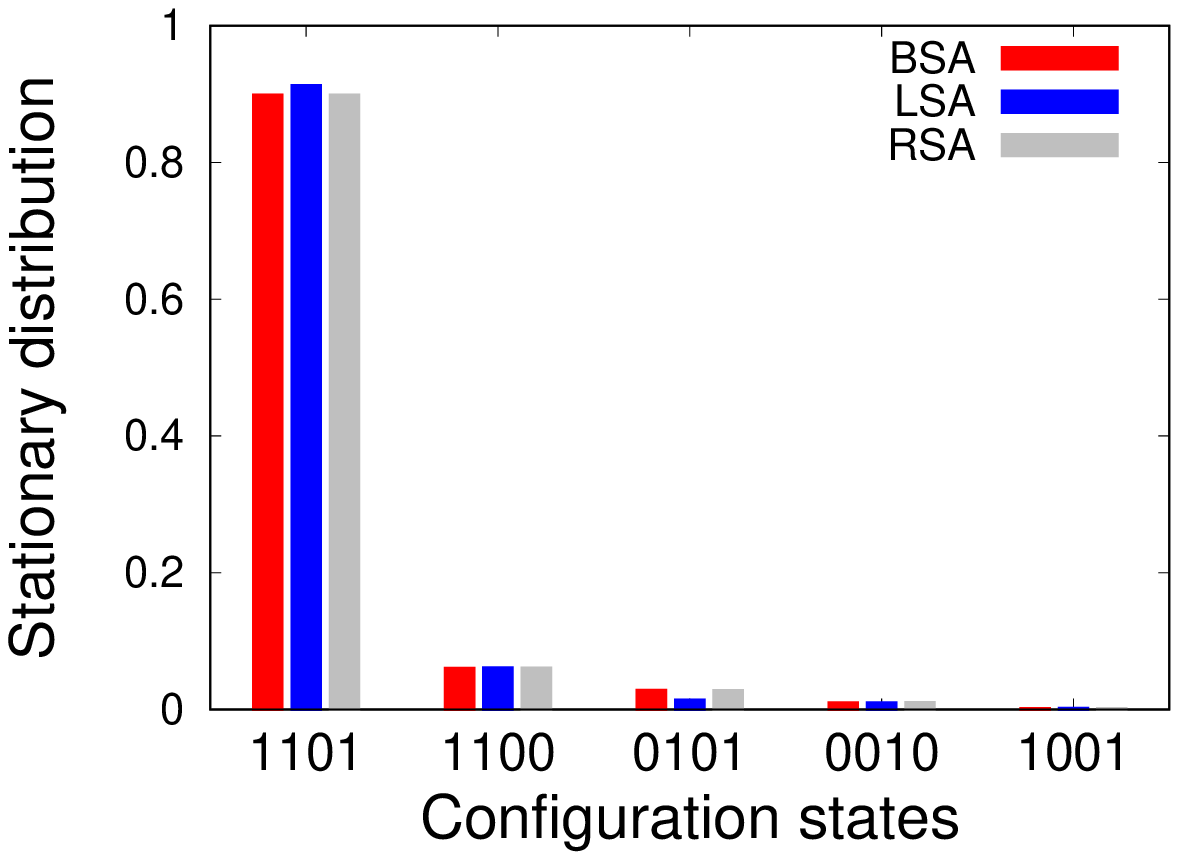} } 
\caption{Comparison of different algorithms in their stationary distributions with different $\beta$.} \label{fig:st_dist} \vspace{-2mm}
\end{figure}

In Fig \ref{fig:var}, we plot the variance rate defined in Eq. (\ref{eq:avr}) with the choice of $h(X(t)) = \idc_c(X(t)) \triangleq \idc\{X(t) = (0,0,1,0)\}$ in order to look at the variability of the cumulative service process of link $c$ over different time scales. 
In this case, $\beta=0.5$ is used.
The figure \ref{fig:var}-(a) shows the result with the message drop probability 0.1, in which the variability of link service process due to LSA algorithm increases quite a bit, whereas the result from RSA algorithm is almost same as that of BSA algorithm.
When the message drop rate is high (as in Fig \ref{fig:var}-(b)), the performance loss due to LSA algorithm is quite significant, whereas RSA algorithm can still be performed \emph{as if there is no message drop} in this case.
These results are consistent with our expectation described in Conjecture \ref{conj:ordering}.

\begin{figure}[t]  \vspace{-2mm}
  \subfloat[Message drop pr. is 0.1]{\includegraphics[width=1.75in]{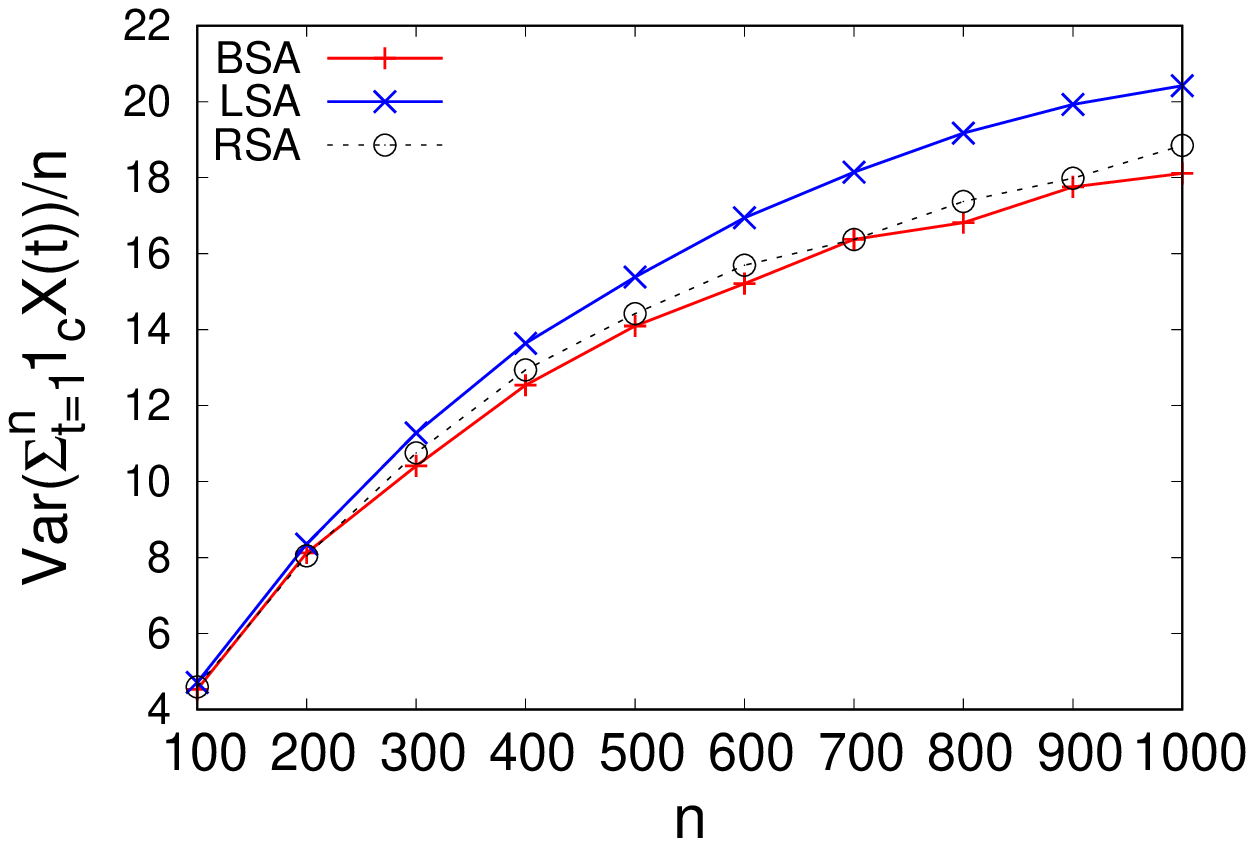} }
  \subfloat[Message drop pr. is 0.5]{\includegraphics[width=1.75in]{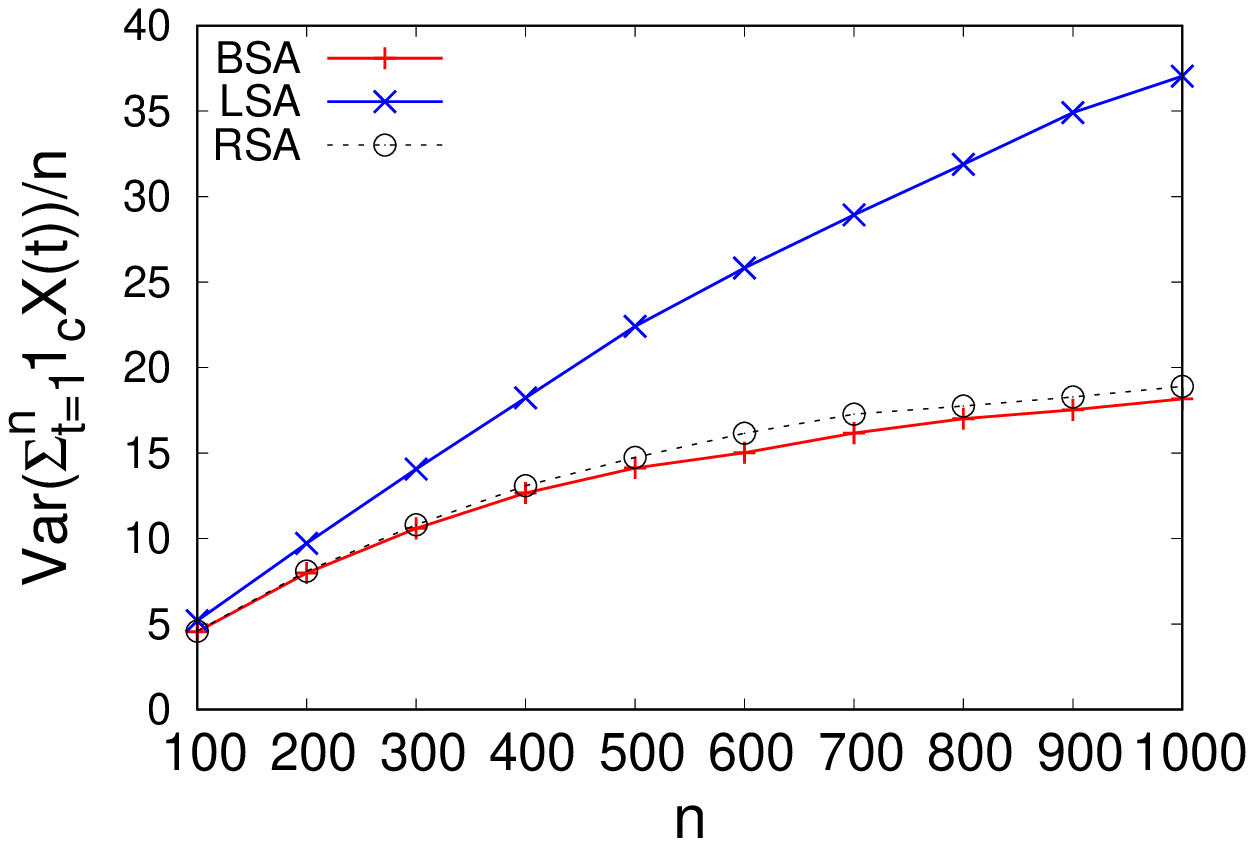} } 
\caption{The variance rate of the link $c$'s service process over different time scales ($\beta=0.5$).} \label{fig:var}  \vspace{-4mm}
\end{figure}

We here consider a queueing application scenario, where each node is associated with queue fed by an external packet arrival process.
Let $Q_i(t)$ be the queue size of node $i$ at time $t$ of which dynamics is determined by the typical queueing process: $Q_i(t) = [Q_i(t-1) + a_i(t) - c_i(\x(t))]^+$, $t \geq 1$, where the arrival process $a_i(t)$ is assumed to be a constant $a_i$ over time. 
In order to adapt to the dynamic queue size, we adopt the popular technique of \emph{dynamic fugacity} scheme used in the CSMA scheduling, in which the weight parameters are chosen such that $W_i(t) = \log(Q_i(t)+1)$\footnote{The log function is used to effectively emulate the time scale separation assumption in a large queue regime which is a standard technique \cite{SRS09, JRJ10CDC}.}. The parameter $\beta$ is simply set to be 1, since setting large weight parameters (or equivalently large queue size) will act as setting large $\beta$ in the max-weight problem, and the appropriate weights will be automatically found from the queue sizes. Now the weight parameters are dynamic over time, and therefore they should also be informed to the corresponding nodes as a form of message, which is also subject to the delivery failure events. 
Since the queue dynamics has a limited evolution over time: $Q_i(t)+a_ik-k \leq Q_i(t+k) \leq Q(t) + a_ik$, $k \geq 0$,
each node $i$ can obtain an upper bounded estimate for the weights $w_j(t+k)$ for $j \in N_i$, based on the most recently observed value of $w_j(t)$. 
We have implemented these schemes in the simulation, and plotted the average queue size of link $c$ in Fig \ref{fig:delay}. The results show that the improvement of average queue size by RSA algorithm can be quite significant especially when the message drop rate is high.

\begin{figure}[t] \vspace{-4mm}
  \subfloat[Message drop pr. is 0.1]{\includegraphics[width=1.75in]{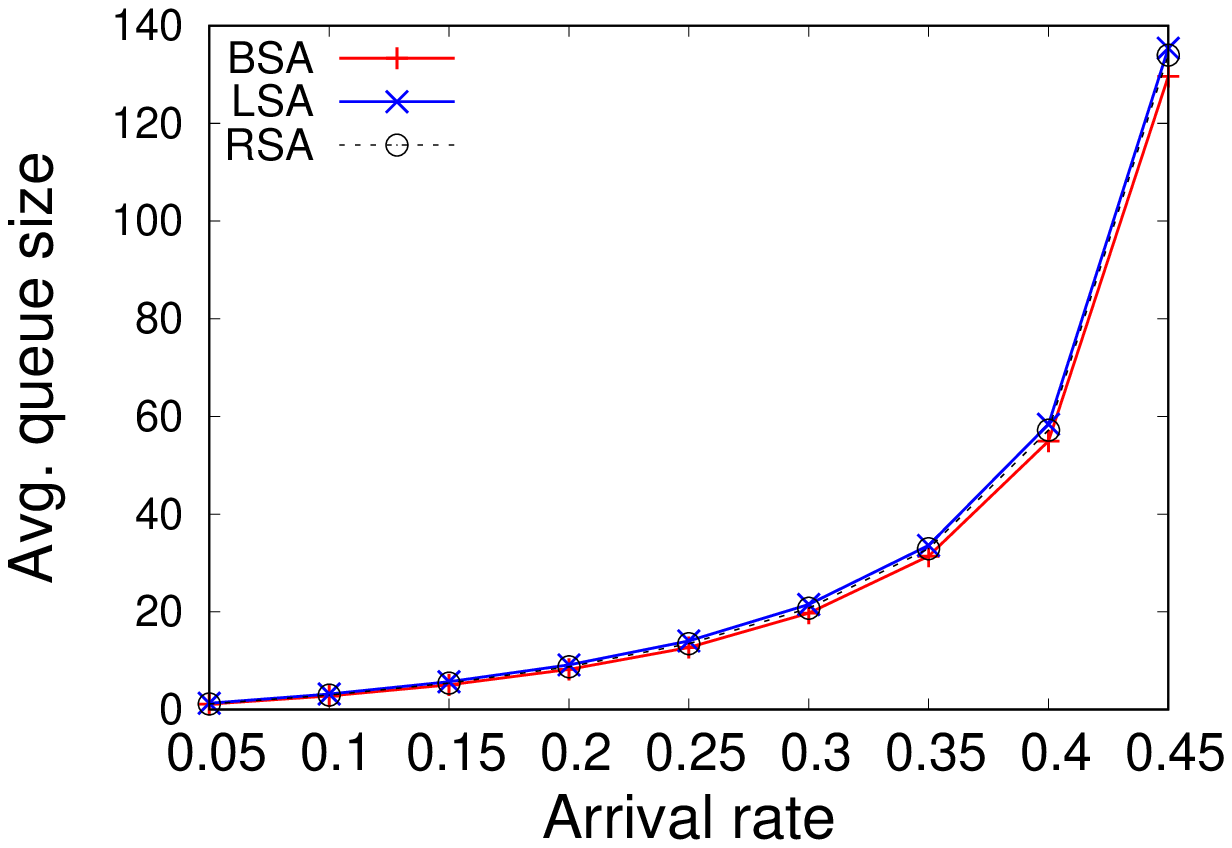} }
  \subfloat[Message drop pr. is 0.5]{\includegraphics[width=1.75in]{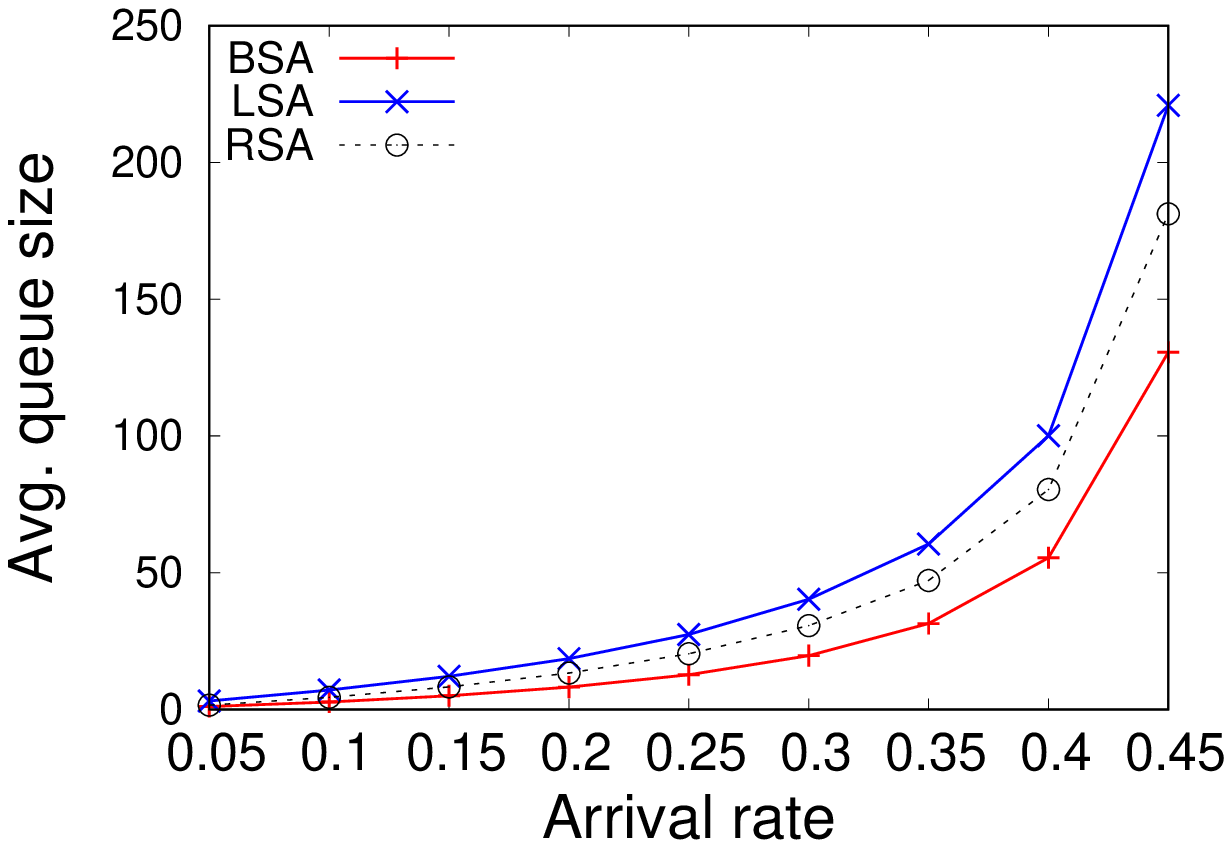} } 
\caption{Average queue size of link $c$} \label{fig:delay}  \vspace{-4mm}
\end{figure}

We also perform simulations with a realistic SINR model in a larger network with 10 pairs of communication links randomly deployed in a $500m\times 500m$ geographical region as shwon in Fig \ref{fig:random_topology}. The parameter settings are as following. The transmitter node of each link $i$ has three transmission power modes, $x_i \in \{0, 5, 10\}$mW, and the transmitted signal experiences pass loss by $d_{ij}^{-\gamma}$, where $\gamma=4$ and $d_{ij}$ is the distance between the transmitter $i$ and the receiver $j$ of the signal, and the thermal noise $n_0 = 10^{-12}$mW is used for all links. Each link obtains different amount of capacity depending on its experienced SINR level as described in Table \Romannum{1}, and the bound parameters are chosen with $c_i^{\max}=3, \forall i$ and in the way we described in Section \Romannum{5}-A. We have selected the set of neighbors from which the messages are to be collected during the messaging phase such that a link $j$ belongs to link $i$'s neighbor set $N_i$ if the receiver node of link $j$ is located within the range of 250m of the transmitter node of link $i$ (denoted as red-dotted lines in Fig \ref{fig:random_topology}). The packets are injected uniformly to every queue with varying rates from $0.05$ to which any queue gets saturated. We observe that the link with red-circled node tends to have the largest queue among all links, and hence we look at its queue size in the following.

The average queue size of the link with the message drop probabilities is plotted in Fig \ref{fig:delay_rand}. We observe that the queueing performance of LSA algorithm is very sensitive to the message drop events, and it can be shown in Fig \ref{fig:delay_rand}-(a) that its queue size quickly grows even with very small drop probability, where the performance of RSA algorithm is fairly close to the one without any message drop event. Fig \ref{fig:delay_rand}-(b) shows that as the drop probability increases, the performance of RSA algorithm also gets worse, however, its improvement is substantial in comparison to LSA algorithm.

We also look at the impact of using loose bounds, rather than using the known-tight bound. 
For this, we intentionally used large values of $c_i^{\max} \in [3,\ldots,10]$ for the bound parameters, and performed the simulations.
Fig \ref{fig:ratio} plots the ratio of the measured queue size between RSA and LSA algorithm with using different bound parameters under the different message drop regimes (from 0.05 to 0.4). 
It shows that in both cases of light and heavy packets arrival intensity, the improvement is remarkable even with using twice larger bound than the tight bound when e.g., drop rate is 0.2 (the ratio in this case is 0.85 and 0.78 for the arrival rate 0.2 and 0.4 respectively), and in general the improvement by RSA algorithm can be seen substantial unless the bound is too loose.

\begin{table}
  \begin{minipage}[b]{0.44\linewidth}
    \centering
    \begin{tikzpicture} [scale=0.006]
\tikzstyle{red_nodes} = [draw=red!75, thick, circle, fill=red!20, minimum size=3mm, scale=0.4]
\tikzstyle{blue_nodes} = [draw=blue!75, thick, circle, fill=blue!20, minimum size=3mm, scale=0.4]
\draw[step=100,gray,very thin] (0,0) grid (500,500);
\node [blue_nodes] (t1) at (110,410) {};
\node [blue_nodes] (r1) at (160,440) {};
\node [blue_nodes] (t2) at (100,290) {};
\node [blue_nodes] (r2) at (78,210) {};
\node [red_nodes] (t3) at (183,301) {};
\node [blue_nodes] (r3) at (258,248) {};
\node [blue_nodes] (t4) at (179,51) {};
\node [blue_nodes] (r4) at (95,71) {};
\node [blue_nodes] (t5) at (380,450) {};
\node [blue_nodes] (r5) at (322,422) {};
\node [blue_nodes] (t6) at (440,310) {};
\node [blue_nodes] (r6) at (400,390) {};
\node [blue_nodes] (t7) at (451,65) {};
\node [blue_nodes] (r7) at (412,120) {};
\node [blue_nodes] (t8) at (270,30) {};
\node [blue_nodes] (r8) at (350,50) {};
\node [blue_nodes] (t9) at (490,260) {};
\node [blue_nodes] (r9) at (413,220) {};
\draw [->] (t1) -- (r1);
\draw [->] (t2) -- (r2);
\draw [->] (t3) -- (r3);
\draw [->] (t4) -- (r4);
\draw [->] (t5) -- (r5);
\draw [->] (t6) -- (r6);
\draw [->] (t7) -- (r7);
\draw [->] (t8) -- (r8);
\draw [->] (t9) -- (r9);
\draw [dotted,red,-] (r2) -- (t1);
\draw [dotted,red,-] (r3) -- (t1);
\draw [dotted,red,-] (r5) -- (t1);
\draw [dotted,red,-] (r1) -- (t2);
\draw [dotted,red,-] (r3) -- (t2);
\draw [dotted,red,-] (r4) -- (t2);
\draw [dotted,red,-] (r1) -- (t3);
\draw [dotted,red,-] (r2) -- (t3);
\draw [dotted,red,-] (r4) -- (t3);
\draw [dotted,red,-] (r5) -- (t3);
\draw [dotted,red,-] (r6) -- (t3);
\draw [dotted,red,-] (r9) -- (t3);
\draw [dotted,red,-] (r2) -- (t4);
\draw [dotted,red,-] (r3) -- (t4);
\draw [dotted,red,-] (r7) -- (t4);
\draw [dotted,red,-] (r8) -- (t4);
\draw [dotted,red,-] (r1) -- (t5);
\draw [dotted,red,-] (r3) -- (t5);
\draw [dotted,red,-] (r6) -- (t5);
\draw [dotted,red,-] (r9) -- (t5);
\draw [dotted,red,-] (r3) -- (t6);
\draw [dotted,red,-] (r5) -- (t6);
\draw [dotted,red,-] (r7) -- (t6);
\draw [dotted,red,-] (r9) -- (t6);
\draw [dotted,red,-] (r8) -- (t7);
\draw [dotted,red,-] (r9) -- (t7);
\draw [dotted,red,-] (r3) -- (t8);
\draw [dotted,red,-] (r4) -- (t8);
\draw [dotted,red,-] (r7) -- (t8);
\draw [dotted,red,-] (r9) -- (t8);
\draw [dotted,red,-] (r3) -- (t9);
\draw [dotted,red,-] (r5) -- (t9);
\draw [dotted,red,-] (r6) -- (t9);
\draw [dotted,red,-] (r7) -- (t9);
\end{tikzpicture} \captionof{figure}{Network topology. Arrows represent communication links, and red-dotted lines indicate the message collection structure.} \label{fig:random_topology}
  \end{minipage}%
  \begin{minipage}[b]{0.60\linewidth}
    \centering
\begin{tabular}[b]{ |c|c| }
  \hline
  \multirow{ 2}{*}{SINR (dB)}  & Data rate \\
  & (units per slot) \\
  \hline \hline
  $(-\infty, 10]$ & 0 \\ \hline
  $(10, 20]$ & 1 \\ \hline
  $(20, 30]$ & 2 \\ \hline
  $(30, \infty)$ & 3\\ \hline
\end{tabular} \vspace{0.5cm} \caption{Data rates as a function of SINR} 
\end{minipage}
\label{tab:sinr}
\end{table}

\begin{figure} [t]  \vspace{-4mm}
  \subfloat[Message drop pr. is 0.1]{\includegraphics[width=1.75in]{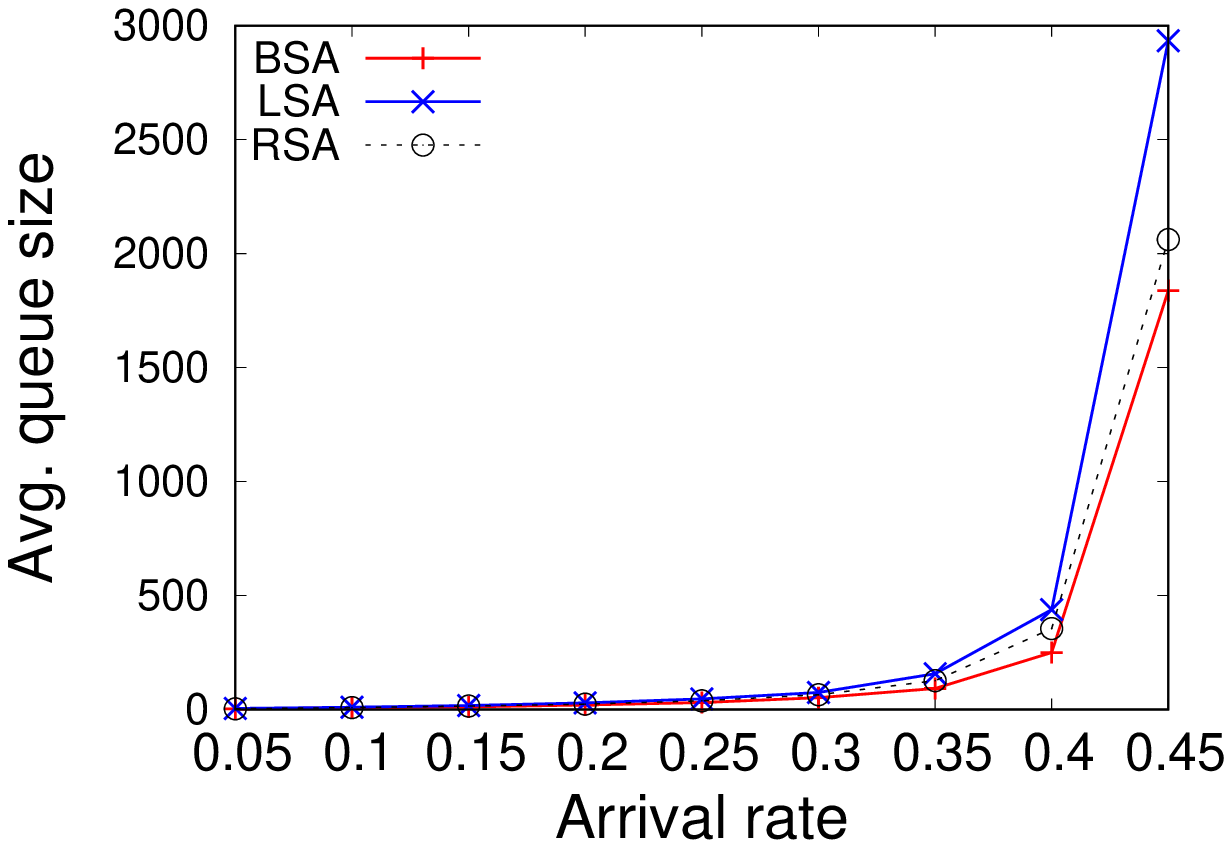} }
  \subfloat[Message drop pr. is 0.3]{\includegraphics[width=1.75in]{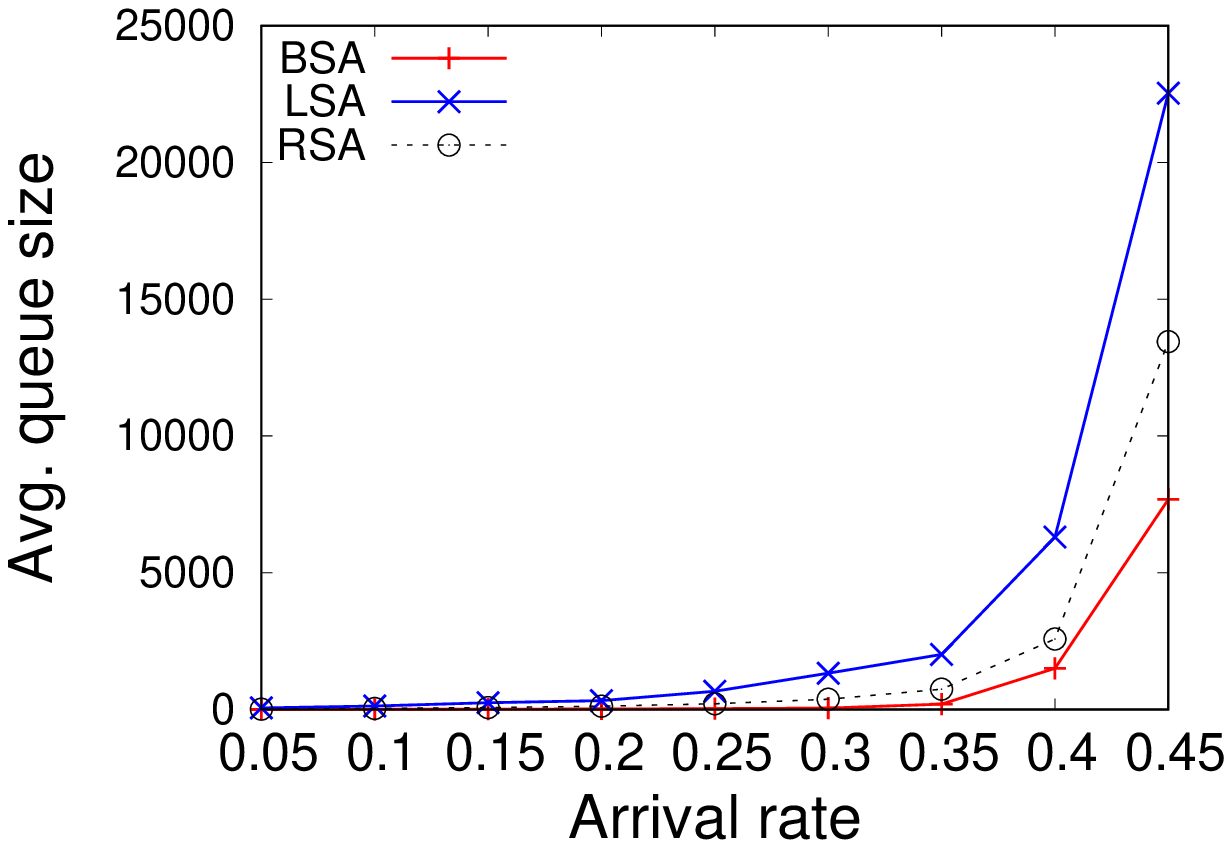} } 
\caption{Average queue size of the red-circled link in Fig \ref{fig:random_topology}.} \label{fig:delay_rand}  \vspace{-3mm}
\end{figure}

\begin{figure} [t] \vspace{-4mm}
  \subfloat[Packet arrival rate is 0.2]{\includegraphics[width=1.75in]{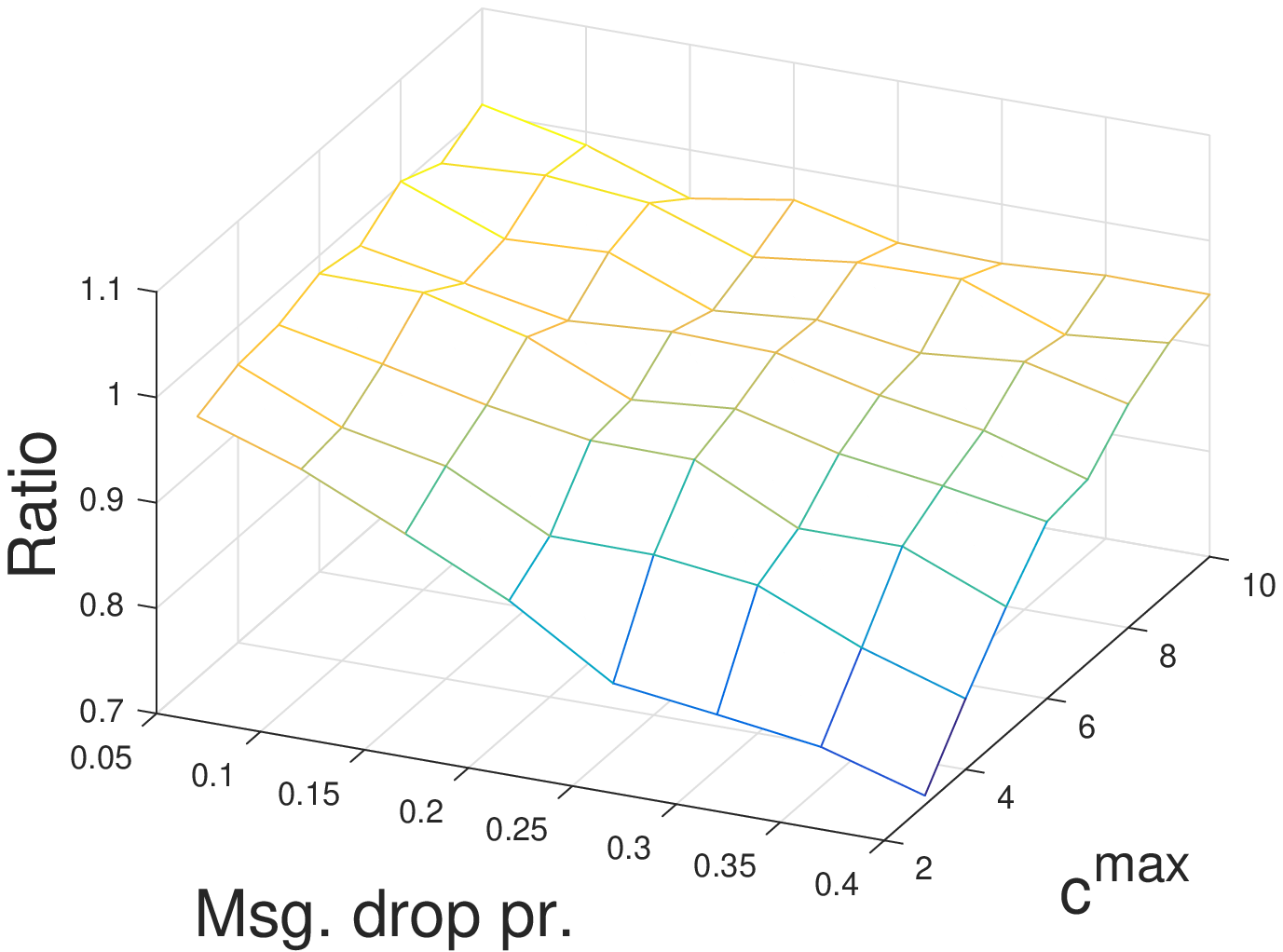} }
  \subfloat[Packet arrival rate is 0.4]{\includegraphics[width=1.75in]{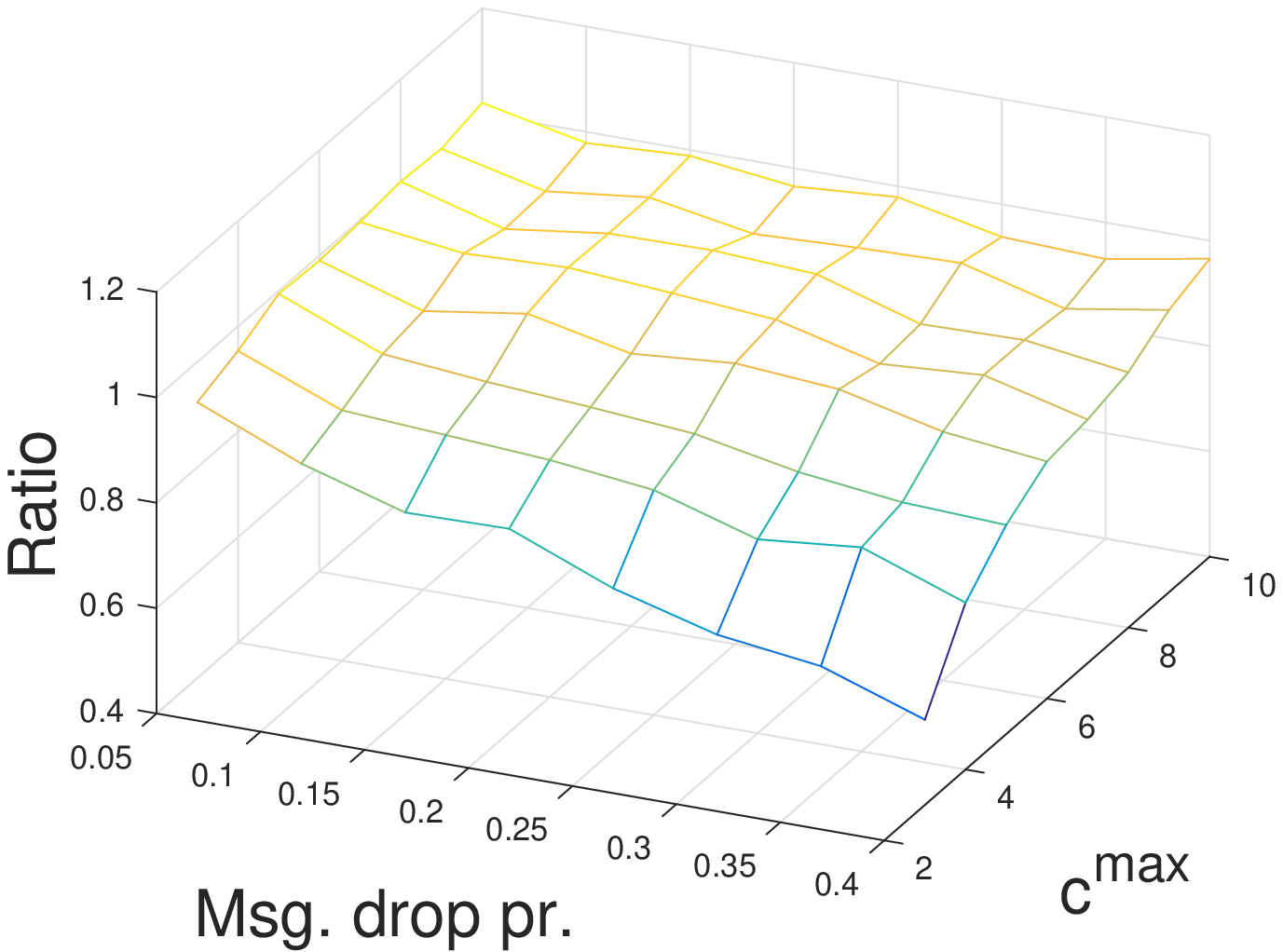} } 
\caption{The ratio of the averge queue size between RSA and LSA algorithms (RSA/LSA) with using different bounds in the various message drop regimes.} \label{fig:ratio}  \vspace{-3mm}
\end{figure}

%% file: 8_conclusion.tex

\section{Discussion}

In this section, we discuss the relationship of our proposed approach with existing solutions from a practical viewpoint.

We remark that our implementation structure is comparable to those in \cite{borst2011distributed} and \cite{qian2010globally}. 
While the goal of achieving the desired stationary distribution is the same, their implementation is a little different from ours.
Their schemes are close to a \emph{proactive} approach in the following sense. Each node evaluates the performance objective of others based on its locally stored variable and prior knowledge on their functionals. Whenever a node changes its state, it proactively broadcasts (within a suitably defined local range) to other nodes to convey the new state, and those local variables can be updated upon receiving the message. 
On the other hand, our approach can be seen as a \emph{reactive} approach in that whenever a node wants to update its state, it sends out a request signal to other nodes, and then those nodes that received the signal reacts to the request by sending messages containing its functional difference as described in our algorithm.

We notice that the proactive approach has a few drawbacks. First, even in the one-hop interference model - a capacity model that only considers interference from nodes within a directly communicable range - the earlier mentioned broadcasting task has to be performed over a two-hop range, except for some special instances \cite{borst2011distributed}. Also, the condition that each node has a prior knowledge on the objective function of others may not be a realistic assumption in practice. In the SINR model, this condition requires that all pairs of links have knowledge of the channel gain terms between the transmitter of their own link and the receiver of their neighboring links, which is difficult to know a priori.

Regarding the message overhead, we take no position on one approach being better than the other, as it will greatly differ depending on the type of network topology and how to realize those message exchange mechanisms with particular communication systems.
Typically, broadcasting is easier than receiving different information from different nodes, however, it has been well known that ensuring the correct reception of broadcasting is a difficult job, which may incur additional overhead. Needless to say, doing that with two-hop broadcasting is even more difficult. 
To avoid this difficulty, \cite{borst2011distributed} and \cite{qian2010globally} suggest using \emph{out-dated} values, i.e., the most recently known values about the corresponding variables. However, the mismatch between the local variables and the actual ones will incur bias to the resulting stationary distribution, and thus whether it can achieve the same optimality is not clear. This fact has been neglected in those works.

In contrast, in our approach, in the one-hop interference model, we only require one-hop information, where its complexity can be robustly handled by the new approach we introduced in this paper. Furthermore, our algorithm does not need to know the precise form of the objective functions.
Hence, our implementation approach is advantageous and should be applicable to communication systems that go even beyond the SINR model.

\section{Conclusion}

In this paper, we investigate important practical considerations and performance issues that arise when the traditional SA is implemented in wireless networks with imperfect communications.
We recognize that various practical factors including the inherently noisy nature of wireless communications and the increasing message complexity in modern communication technologies can be critical sources that prevent the efficient realization of SA in general wireless networks. 
Our simulation results show that a straightforward solution to bypass this problem is not practical due to its slow operation speed. 
To tackle this problem, we propose a novel approach that allows the algorithm to operate with only partial observations on the system performance objective, which helps improve its computational speeds. 
We rigorously show that the new algorithm exhibits the same convergence in probability to the optimal states under the standard annealing technique.


%% file: appendix.tex
\appendix

In this appendix, we provide proofs for Proposition \ref{lem:LSA_equiv}, \ref{prop:ordering}, and Theorem \ref{thm:optimal}.
To that end, we first write down the transition probabilities of BSA, LSA and RSA algorithms, respectively denoted by $\bP^{B}$, $\bP^{L}$, and $\bP^{R}$, to be referred in the proofs. For notational simplicity, we denote for two configurations $\x,\x' \in \bX$ which differ only at one node $i \in \N$, i.e., $x_i \neq x'_i$ and $x_j = x'_j$ for all $j \in \N \backslash \{i\}$, that $\hat{i}(\x,\x')$ (or simply $\hat{i}$ when its definition is clear) indicates the node with the different state.

The transition probability of BSA algorithm, $P^{B}(\x,\x')$, for $\x,\x' \in \bX \ (\x \neq \x')$, is
\begin{eqnarray} 
P^{B}(\x,\x') = c(\x,\x') e^{- \beta [-\Delta(\x,\x')]^+}, \label{eq:tran_p_bsa}
\end{eqnarray}
where $\Delta(\x,\x') = \sum_{j \in N_{\hat{i}} \cup \{\hat{i}\}} f_j(\x') - f_j(\x)$ and
\begin{align*}
&c(\x,\x') =
\begin{cases} \frac{1}{n(|\M|-1)}, &\begin{array}{@{}l@{}}\text{if for some $i \in \N$,}\  x_i = x'_i \ \text{and},\\ x_j \neq x'_j, \forall j \in N_i \backslash \{i\} \end{array}\\
   0, &\mbox{if otherwise.}
\end{cases} 
\end{align*}

The transition probability of LSA algorithm, $P^{L}(\x,\x')$, for $\x,\x' \in \bX \ (\x \neq \x')$, is
\begin{eqnarray}
P^{L}(\x,\x') = c(\x,\x') q_{\hat{i}, N_{\hat{i}}}  e^{- \beta [-\Delta(\x,\x')]^+}. \label{eq:tran_p_lsa}
\end{eqnarray}

The transition probability of RSA algorithm $P^{R}(\x,\x')$, for $\x,\x' \in \bX \ (\x \neq \x')$, is
\begin{eqnarray}
P^{R}(\x,\x') = c(\x,\x') \sum_{S \subseteq N_{\hat{i}}} q_{\hat{i},S} e^{-\beta [-\Delta_{[S]}(\x,\x')]^+} \label{eq:tran_p_rsa}
\end{eqnarray}
where 
\begin{equation*}
    \Delta_{[S]}(\x,\x') = \sum_{j \in S\cup \{\hat{i}\}} f_j(\x') - f_j(\x) + \sum_{j \in N_{\hat{i}} \backslash \{S\}} b^{\hat{i}j}_{x_{\hat{i}},x'_{\hat{i}}}.
\end{equation*}

For all three algorithms, their self transition probabilities are obtained by $P^{B}(\x,\x) = 1 - \sum_{\x'\neq\x}P^{B}(\x,\x')$ (similarly for $P^{L}$ and $P^{R}$), for all $\x \in \bX$.

\hfill

\textbf{Proof of Proposition \ref{lem:LSA_equiv}. }
\begin{proof} Let $\PI^{L}$ be the stationary distribution of LSA algorithm.
Then, for any $\x, \x' \in \Omega \; (\x \neq \x')$, it holds
\begin{align*}
&&\pi^{L}(\x) P^{L}(\x,\x') &= \pi^{L}(\x') P^{L}(\x',\x)& \\
&\Leftrightarrow& \pi^{L}(\x) P^{B}(\x,\x') &= \pi^{L}(\x') P^{B}(\x',\x).&
\end{align*}
Since $\PI$ is the unique solution of $\PI^L$ in the above set of equations along with the probability constraint, $\sum_{\x} \pi^L(\x) = 1$, it follows $\PI = \PI^{L}$.
\end{proof}

\hfill 

\textbf{Proof of Proposition \ref{prop:ordering}. }

\begin{proof}
Note that for any $\x,\x' \in \bX$ in which only one node state is different, and for any $S \subseteq N_{\hat{i}}$, the following holds 
\begin{align}
& \Delta(\x,\x') = \Delta_{[N_{\hat{i}}]}(\x,\x') \nonumber \\
&= \sum_{j \in N_{\hat{i}} \cup \{\hat{i}\}} f_j(x'_{\hat{i}}, \x_{[-\hat{i}]}) - f_j(x_{\hat{i}}, \x_{[-\hat{i}]}) \nonumber \\
&\geq \sum_{j \in S \cup \{\hat{i}\}} f_j(x'_{\hat{i}}, \x_{[-\hat{i}]}) - f_j(x_{\hat{i}}, \x_{[-\hat{i}]}) \nonumber \\
& \;\;\; + \sum_{j \in N_{\hat{i}} \backslash \{S\}} \left( \min_{\x_{[-i]}} f_j(x'_{\hat{i}}, \x_{[-\hat{i}]}) - \max_{\x_{[-\hat{i}]}} f_j(x_{\hat{i}}, \x_{[-\hat{i}]}) \right)  \nonumber  \\
&\geq \sum_{j \in S \cup \{\hat{i}\}} f_j(x'_{\hat{i}}, \x_{[-\hat{i}]}) - f_j(x_{\hat{i}}, \x_{[-\hat{i}]}) + \sum_{j \in N_{\hat{i}} \backslash S} b^{\hat{i}j}_{x_{\hat{i}},x'_{\hat{i}}} \nonumber \\
&= \Delta_{[S]}(\x,\x'), \label{eq:subset_inequality}
\end{align} 
and therefore we have
\begin{align}
    e^{-\beta[-\Delta_{[S]}(\x,\x')]^+} \leq e^{-\beta[-\Delta(\x,\x')]^+}. \label{eq:subset_inequality2}
\end{align}
Since $q_{i,N_i} \leq \sum_{S \subseteq N_i} q_{i,S} = 1$ holds for all $i \in \N$, the statement follows from the transition probabilitiy of each algorithm in Eq. (\ref{eq:tran_p_bsa}-\ref{eq:tran_p_rsa}) and the inequality in Eq. (\ref{eq:subset_inequality2}).
\end{proof}

\hfill

\textbf{Proof of Theorem \ref{thm:optimal}.}

\begin{proof}
Our proof for the theorem is based on the technique introduced in \cite{connors1989simulated}, in which the annealing optimality of the original SA is proven. We first briefly overview the major steps therein, and apply them to show the optimality of RSA algorithm.

The authors in \cite{connors1989simulated} consider a class of Markov chains whose transition probabilities can be written as a form of Eq. (\ref{eq:trans_pr}), in which the conventional simulated annealing algorithm can be represented by setting $V_{ij} = [f(j) - f(i)]^+$ and $\epsilon(t) = e^{-\beta(t)}$ for achieving minimum $f(\cdot)$.
For this class of Markov chains, they define the \emph{recurrence order} for each state and transition of the Markov chain as follows.

\begin{definition}
The order of recurrence of a state $i \in \Omega$, denoted $\alpha_i$, is
\begin{eqnarray*}
\alpha_i := \begin{cases} -\infty, &\mbox{if}~\sum_{t=1}^{\infty} \mu_i(t) < \infty, \\
  p^{-}, &\mbox{if}~p=\sup\left\{ c \geq 0 : \sum_{t=1}^{\infty} \epsilon(t)^c \mu_i(t) = \infty \right\}\\
  & ~\mbox{and}~\sum_{t=1}^{\infty}\epsilon(t)^p \mu_i(t) < \infty, \\
  p &\mbox{if}~p=\max\left\{ c \geq 0 : \sum_{t=1}^{\infty} \epsilon(t)^c \mu_i(t)= \infty \right\}
\end{cases} 
\end{eqnarray*}
\end{definition}
where $\mu_i(t) = \pr\{X(t) = i\}$ and $p$ is regarded as strictly larger than $p^-$, i.e., $p > p^- > p - \delta_0$ for some $\delta_0 > 0$.
Similarly, the \emph{order of recurrence of the transition from $i$ to $j$} is defined by,

\begin{definition}
The order of recurrence of the transition from $i$ to $j$, denoted $\alpha_{ij}$, is
\begin{eqnarray*}
\alpha_{ij} := \begin{cases} -\infty, &\mbox{if}~\sum_{t=1}^{\infty} \mu_{ij}(t) < \infty, \\
  p^{-}, &\mbox{if}~p=\sup\left\{ c \geq 0 : \sum_{t=1}^{\infty} \epsilon(t)^c \mu_{ij}(t) = \infty \right\}\\
  & ~\mbox{and}~\sum_{t=1}^{\infty}\epsilon(t)^p \mu_{ij}(t) <  \infty, \\
  p &\mbox{if}~p=\max\left\{ c \geq 0 : \sum_{t=1}^{\infty} \epsilon(t)^c \mu_{ij}(t)= \infty \right\}
\end{cases} 
\end{eqnarray*}
\end{definition}
where $\mu_{ij}(t) = \pr\{X(t)=i, X(t+1)=j\}$.
They also defined $\rho$, the \emph{order of cooling} of $\{\epsilon(t)\}$, as follows.
\begin{definition}
The order of the cooling schedule $\{\epsilon(t)\}$, denoted $\rho$, is defined as
\begin{eqnarray*}
\rho := \begin{cases} -\infty, &\mbox{if}~\sum_{t=1}^{\infty} \epsilon(t) < \infty, \\
  p^{-}, &\mbox{if}~p=\sup\left\{ c \geq 0 : \sum_{t=1}^{\infty} \epsilon(t)^c  = \infty \right\} \\
  & ~\mbox{and}~\sum_{t=1}^{\infty}\epsilon(t)^p  < \infty, \\
  p &\mbox{if}~p=\max\left\{ c \geq 0 : \sum_{t=1}^{\infty} \epsilon(t)^c = \infty \right\}
\end{cases} 
\end{eqnarray*}
\end{definition}

Having defined the above terms, we summarize the main results established in \cite{connors1989simulated} as well as in \cite{connors1988balance, desai1994quasi}, which are valid under the following mild assumptions.

\textbf{Assumptions}
\begin{enumerate}
    \item $d$ is sufficiently large. In particular, $d \! \geq  \! 2\! \sum_{(i,j) | V_{i\!j} < \infty} \!\! V_{ij}$.
    \item $\exists \; j \in N_i \; \Leftrightarrow \; \exists \; i \in N_j \; \forall i,j \in \Omega$.
\end{enumerate}

\begin{lemma} \label{lem:org_sim} \cite{connors1989simulated, connors1988balance, desai1994quasi} Under the above assumptions, the followings hold.
\begin{enumerate} 
\item The relation between $\alpha_i$ and $\alpha_{ij}$ is
    \begin{align} \label{eq:prev_results1}
        \alpha_{ij} = \alpha_i - V_{ij} ~~\mbox{for all}~i,j \in \Omega \ (i\neq j),
    \end{align}
\item There is a balance of recurrence orders across every edge in the graph of the Markov chain such that
    \begin{align}
        \max_{i \in A, j \in A^c} \alpha_{ij} = \max_{i \in A, j \in A^c} \alpha_{ji} ~~\mbox{for all}~A \subseteq \Omega. \label{eq:order_balance}
    \end{align}
\item $\{\alpha_i\}$ is the unique solution $\{\lambda_i\}$ of
    \begin{align*} 
        \max_{i \in A, j \in A^c} \lambda_i - V_{ij} &= \max_{i \in A, j \in A^c} \lambda_j - V_{ji} ~~\mbox{for all}~A \subseteq \Omega, \\
        \max_{i \in \Omega} \lambda_i &= \rho.
    \end{align*}
\item Recall that $d$ is the rate of the cooling schedule $\epsilon(t) = t^{-1/d}$, $t \geq 1$. It can be shown that
\begin{equation*}
d = \max_{i \in \Omega} \; \alpha_i = \rho, \text{ and } \alpha_i = \rho \text{ iff } i \in \Omega^*,
\end{equation*}
where $\Omega^* = \{i \in \Omega: f(i) = \min_{j \in \Omega} f(j) \}$.
\item Let $\Omega^{\dagger}$ be the set of states of the largest recurrence order, i.e., $\Omega^{\dagger} := \{ i \in \Omega : \alpha_i = \max_{j\in \Omega} \alpha_j\}$. Then\footnote{In \cite{connors1989simulated, connors1988balance}, the convergence analysis is only conducted for showing $\limsup_{T \rightarrow \infty} \frac{1}{T} \sum_{t=1}^T \pr\{X(t) \in \Omega^{\dagger }\} = 1$, for a general class of cooling schedules, however it can be shown that the limit holds for the particular cooling schedule $\epsilon(t) = t^{-1/d}$ \cite{desai1994quasi}.},
    \begin{align}
        \lim_{T \rightarrow \infty} \frac{1}{T} \sum_{t=1}^T \pr\{X(t) \in \Omega^{\dagger }\} = 1. \label{eq:prev_results3}
    \end{align}
\end{enumerate}
\end{lemma} 

Note that the transition probability of RSA algorithm in Eq. (\ref{eq:tran_p_rsa}) cannot be represent by Eq. (\ref{eq:trans_pr}).
To pursue the above approach to verifying the optimality of RSA algorithm, it is necessary to consider a more general class of Markov chains which has the transition probabilities of the form,
\begin{align} \label{eq:new_form}
    p_{ij}(t) = c_{ij} \sum_{m \in M^{ij}} a^m_{ij} \epsilon(t)^{V^{m}_{ij}}, 
\end{align}
where $M^{ij}$ is some finite set associated with the transition from $i$ to $j$, $i,j \in \Omega$, and $a^m_{ij} \in (0,1]$ is a probability of an element $m \in M^{ij}$ defined over its sample space $M^{ij}$, and $V^{m}_{ij} \geq 0$ for all $m \in M^{ij}$.
Now the transition probability of RSA algorithm can be represented by Eq. (\ref{eq:new_form}) by replacing the corresponding terms: $\Omega \Leftrightarrow \bX$, $i,j \in \Omega \Leftrightarrow \x,\x' \in \bX$, $c_{ij} \Leftrightarrow c(\x,\x')$, $M^{ij} \Leftrightarrow S(\x,\x') := \{ S \in \mathcal{P}(N_{i'}) : q_{\hat{i}, S} > 0\}$, $m \in M^{ij} \Leftrightarrow S \in S(\x,\x')$, $a^m_{ij} \Leftrightarrow q_{\hat{i}, S}$, $V^m_{ij} \Leftrightarrow V^S(\x,\x') := \Delta_{[S]}(\x,\x')$, where $\mathcal{P}(A)$ is the powerset of a set $A$.
For this form, we obtain the generalized correspondence of eq. (\ref{eq:prev_results1}) as stated in the following lemma.

\begin{lemma} \label{lem:general}
$\alpha_{ij} = \alpha_{i} - \max_{m \in M^{ij}} V^m_{ij}, \;\; \forall i,j \in \Omega.$
\end{lemma}
\begin{proof}
By the Chapman-Kolmogorov equation, we obtain
\begin{equation*}
    \mu_{ij}(t) = c_{ij} \sum_{m \in M^{ij}} a_{ij}^m \epsilon(t)^{V^{m}_{ij}} \mu_i(t)
\end{equation*}
and observe that 
\begin{eqnarray*}
\sum_{t=1}^{\infty} \sum_{m\in M^{ij}} a_{ij}^m \epsilon(t)^{c+V^m_{ij}} \mu_i(t) = \infty
\end{eqnarray*}
holds if and only if there exists some $m \in M^{ij}$ such that
\begin{eqnarray*}
\sum_{t=1}^{\infty} \epsilon(t)^{c+V^m_{ij}} \mu_i(t) = \infty.
\end{eqnarray*}
Therefore, we have
\begin{eqnarray*}
&\sup\{ c \geq 0 : \sum_{t=1}^{\infty} \sum_{m \in M^{ij}} a_{ij}^m \epsilon(t)^{c+V^m_{ij}} \mu_i(t) = \infty \} \\
&= \sup\{ c \geq 0 : \sum_{t=1}^{\infty} \epsilon(t)^{c+ \max_{m \in M^{ij}} V^m_{ij}} \mu_i(t) = \infty \}
\end{eqnarray*}
from which the result follows by the definitions of $\alpha_i$ and $\alpha_{ij}$.
\end{proof}

Let $\alpha(\x)$ and $\alpha(\x,\x')$ be the recurrence order of state $\x$ and that of state transition from $\x$ to $\x'$, for $\x,\x' \in \Omega \ (\x \neq \x')$ due to the Markov chain induced by BSA Algorithm, and let $\hat{\alpha}(\x)$ and $\hat{\alpha}(\x,\x')$ be those of the chain from RSA algorithm, respectively. In the following lemmas, we verify that all the recurrence orders for any state and transition are equivalent between the two algorithms.

\begin{lemma} \label{lem:max_equiv}
$\max_{\x \in \bX} \hat{\alpha}(\x) = \rho = d$.
\end{lemma}
\begin{proof}
Note that $\sum_{t = 1}^{\infty} \epsilon(t)^c < \infty$ for $c > d$ and $\sum_{t = 1}^{\infty} \epsilon(t)^c = \infty$ for $c \leq d$, and thus by the definition of $\rho$, $d = \rho$. On the other hand,
\begin{equation*}
    \sum_{t = 1}^{\infty} \epsilon(t)^d = \sum_{\x \in \bX} \left( \sum_{t=1}^{\infty} \epsilon(t)^d \mu_{\x}(t) \right) = \infty,
\end{equation*}
where $\mu_{\x}(t) = \pr\{X(t) = \x\}$, implying $ \sum_{t=1}^{\infty} \epsilon(t)^d \mu_{\x}(t) = \infty$ for some $\x \in \bX$. Therefore, $\max_{\x \in \bX} \alpha(\x) = d$, and the same analysis also applies to $\hat{\alpha}(\x)$.
\end{proof}

\begin{lemma} \label{lem:equiv}
$\alpha(\x) = \hat{\alpha}(\x), \; \text{and} \; \alpha(\x,\x') = \hat{\alpha}(\x,\x')$, for all $\x,\x' \in \Omega$, 
\end{lemma}
\begin{proof}
Note that the order balance equation of Eq. (\ref{eq:order_balance}) for the BSA algorithm can be written by
\begin{align*}
    \max_{\x \in A, \x' \in A^c} \alpha(\x,\x') = \max_{\x \in A, \x' \in A^c} \alpha(\x',\x), \; \forall A \subseteq \bX,
\end{align*}
and $\{\alpha(\x)\}$ is the unique solution of $\{\lambda(\x)\}$ of the equations
\begin{align}
    \max_{\x \in A, \x' \in A^c} \lambda(\x) - V^{N_{\hat{i}}}(\x,\x') = \max_{\x \in A, \x' \in A^c} \lambda(\x') - V^{N_{\hat{i}}}(\x',\x), 
    \label{eq:lsa_balance}
\end{align}
for all $A \subseteq \bX$ along with $\max_{\x \in \bX} \lambda(\x) = \rho$.
On the other hand, a similar characterization of $\{\hat{\alpha}(\x)\}$ of RSA algorithm using the order balance equation,
\begin{equation*}
    \max_{\x \in A, \x' \in A^c} \hat{\alpha}(\x,\x') = \max_{\x \in A, \x' \in A^c} \hat{\alpha}(\x',\x), \; \forall A \subseteq \bX,
\end{equation*}
can be written (from Lemma \ref{lem:general}) as the solution $\{\hat{\lambda}(\x)\}$ of
\begin{align}
    & \max_{\x \in A, \x' \in A^c} \left( \hat{\lambda}(\x) - \max_{S \subseteq S(\x,\x')} V^{S}(\x,\x') \right)  \nonumber \\
    & = \max_{\x \in A, \x' \in A^c} \left( \hat{\lambda}(\x') - \max_{S \subseteq S(\x',\x)} V^{S}(\x',\x)\right), \label{eq:balance_eq_rsa}
\end{align}
for all $A \subseteq \bX$ with $\max_{\x \in \bX} \hat{\lambda}(\x) = \rho$ (due to Lemma \ref{lem:max_equiv}).
Observe that for any $S \subseteq S(\x,\x')$, $\x,\x' \in \bX \ (\x\neq \x')$, it holds
\begin{align*}
& V^{N_{\hat{i}}}(\x,\x') = \Delta_{[N_{\hat{i}}]}(\x,\x') \geq  \Delta_{[S]}(\x,\x') = V^{S}(\x,\x'),
\end{align*}
where the inequality is from Eq. (\ref{eq:subset_inequality}).
Since $q_{i,N_i} > 0$ for all $i \in \N$, we have 
\begin{equation*} 
 V^{N_{\hat{i}}}(\x,\x') = \max_{S \in S(\x,\x')} V^S(\x,\x'),
\end{equation*}
and hence, 
\begin{align}
    \hat{\alpha}(\x,\x') &= \hat{\alpha}(\x) - \max_{S \in S(\x,\x')} V^{S}(\x,\x')  \nonumber \\
    &= \hat{\alpha}(\x) - V^{N_{\hat{i}}}(\x,\x'), \label{eq:V_rsa}
\end{align}
for all $\x, \x' \in \bX (\x \neq \x')$, which presents the same forms in Eq. (\ref{eq:lsa_balance}).
Therefore, Lemma \ref{lem:org_sim}-3) implies that the (unique) solution for Eq. (\ref{eq:balance_eq_rsa}) is identical to that of Eq. (\ref{eq:lsa_balance}), i.e., $\alpha(\x) = \hat{\alpha}(\x)$ for all $\x \in \bX$. Also, $\alpha(\x,\x') = \hat{\alpha}(\x,\x')$ holds for all $\x,\x' \in \Omega \ (\x \neq \x')$ from eq. (\ref{eq:V_rsa}).
\end{proof}

The result follows from the above lemma and Lemma \ref{lem:org_sim}-4) and \ref{lem:org_sim}-5).
\end{proof}